\documentclass[11pt]{amsart}
\usepackage{hyperref}
\usepackage{tabularx,booktabs}
\usepackage{caption}
\usepackage{amsmath}
\usepackage{amsfonts}

\usepackage{amscd}
\usepackage{amsthm}
\usepackage{amssymb}
\usepackage{latexsym,float}
\usepackage{euscript}
\usepackage{epsfig}
\usepackage{graphics}
\usepackage{array}
\usepackage{enumerate}
\usepackage{color,tikz}
\usetikzlibrary{arrows,chains,matrix,positioning,scopes}
\makeatletter
\tikzset{join/.code=\tikzset{after node path={%
\ifx\tikzchainprevious\pgfutil@empty\else(\tikzchainprevious)%
edge[every join]#1(\tikzchaincurrent)\fi}}}
\makeatother
\tikzset{>=stealth',every on chain/.append style={join},
         every join/.style={->}}
\tikzstyle{labeled}=[execute at begin node=$\scriptstyle,
   execute at end node=$]
\usepackage{wasysym}
\usepackage{hyperref}
\usepackage{pdfsync}
\usepackage{comment}
\usepackage[all]{xy}

\allowdisplaybreaks[4]

\renewcommand{\vec}[1]{\mbox{\boldmath$#1$}}

\newcommand{\bel}[1]{\begin{equation}\label{#1}}

\newcommand{\be}{\begin{equation}}
\newcommand{\sign}{\ensuremath{\mathrm{sign}}}

\newcommand{\videpost}{{vide post}{}}

\newcommand{\power}{\ensuremath{\mathcal{P}}}
\newcommand{\vol}{\ensuremath{\mathrm{vol}}}

\newcommand{\ah}{\ensuremath{h_{\mathrm{anti}}}}

\newcommand{\Hmm}[1]{\leavevmode{\marginpar{\tiny%
$\hbox to 0mm{\hspace*{-0.5mm}$\leftarrow$\hss}%
\vcenter{\vrule depth 0.1mm height 0.1mm width \the\marginparwidth}%
\hbox to
0mm{\hss$\rightarrow$\hspace*{-0.5mm}}$\\\relax\raggedright #1}}}

\newtheorem{theorem}{Theorem}[section]

\newtheorem{lemma}[theorem]{Lemma}
\newtheorem{corollary}[theorem]{Corollary}
\newtheorem{definition}[theorem]{Definition}

\newtheorem{remark}[theorem]{Remark}

\newtheorem{prop}[theorem]{Proposition}

\newtheorem{question}[theorem]{Question}

\newcommand{\norm}[1]{\Vert #1\Vert_\infty}
\newcommand{\set}[1]{\{ #1\}}

\newcommand{\abs}[1]{\vert #1\vert}

\DeclareMathOperator*{\argmin}{\ensuremath{arg\,min}}

\begin{document}

\title{Lov\'asz extension and graph cut}

\author{Kung-Ching Chang}
\email{kcchang@math.pku.edu.cn}

\author{Sihong Shao}
\email{sihong@math.pku.edu.cn}

\author{Dong Zhang}
\email{dongzhang@pku.edu.cn}

\author{Weixi Zhang}
\email{zhangweixi@pku.edu.cn}

\address{LMAM and School of Mathematical Sciences,  Peking University, Beijing 100871, China}

\begin{abstract}
A set-pair Lov\'asz extension is established to construct equivalent continuous optimization problems for
graph $k$-cut problems.

%graph cut including the Cheeger cut, maxcut, dual Cheeger cut, max 3-cut, anti-Cheeger cut problems, as well as general $k$-cut problems.

%With the help of the explicit objective functions of resulting  continuous optimization, we continue to develop the spectral theory of graph 1-Laplacian for both maxcut and dual Cheeget cut problems as an illustration. The related eigenvalue problem is to solve a nonlinear system involving a set-valued function. Elementary facts and rich structure in nodal domains for the spectrum of graph 1-Laplacian are fully revealed. Particularly,
%we deduce the equality form relation between some eigenvalues of graph 1-Laplacian and the optimal values, and meanwhile connect corresponding eigenvectors to the optimal cuts.
\end{abstract}
\maketitle
       \makeatletter
\@addtoreset{equation}{section}
\makeatother   

\tableofcontents

Keywords: 
Lov\'asz extension;
Submodular;
Dual Cheeger cut;
Maxcut;
Graph cut
%Anti-Cheeger cut

Mathematics Subject Classification: 90C35, 90C27, 05C85, 58E05.

%\author{\large  Bobo Hua$^\ast$\ \ \ \  Yanhui Su$^{\dag}$} % ����??��???

%\par
%\maketitle

%\bigskip

%BB: \begin{enumerate}%\item We denote by the distance function by $d$, not by $d^G.$
%%\item Euler characteristic of open surfaces, how to define???
%\item Check citations, number of theorems!
%%\item Let be, be.
%%\item Without loss of generality, we may always consider planar graphs with nonnegative curvature.
%%\item $\partial G$ to $\partial S.$
%\end{enumerate}

\section{Introduction}
\label{sec:intro}

%the maxcut problem \cite{Karp1972}, and

Motivated by the need of practical application in e.g., machine learning and big data, it is not only natural but also imperative for applied mathematicians to plug into valuable subjects emerged from well-established mathematics such as analytic techniques, topological views and algebraic structures. A firm bridge between discrete data world and continuous math field should be tremendously helpful. Along this direction, the Lov\'asz extension \cite{Lovasz1983} provides a both explicit and equivalent continuous optimization problem for a discrete optimization problem.
However, it deals with set-functions which admit only one input set and thus correspond to so-called $2$-cut problems, for instance, the Cheeger cut problem \cite{BuhlerRangapuramSetzerHein2013}.
Therefore it is not straightforward to apply the original Lov\'asz extension into general $k$-cut problems, such as the dual Cheeger cut \cite{Trevisan2012,BauerJost2013}.
Accordingly, we ask
\begin{question}\label{question:graphcut}
 How to write down a both explicit and equivalent continuous optimization problem for a graph $k$-cut problem?
\end{question}

Let us introduce some notations first.  
$G=(V,E)$ is an unweighted and undirected graph with vertex set $V=\{1,2,\cdots,n\}$ and edge set $E$, and $w_{ij}$ the weight of the edge $i\sim j$. For two disjoint subsets $A$ and $B$ of $V$, let $E(A,B)$ denote the set of edges that cross $A$ and $B$. For $S\subset V$, let $S^c=V\backslash S$ be the complement of $S$. The edge boundary of $S$ is $\partial S=\partial S^c=E(S, S^c)$.  The amount of edge set $E(A,B)$ is denoted by $|E(A,B)|=\sum_{i\in A}\sum_{j\in B}w_{ij}$, and  the volume of $S$ is defined to be $ \vol(S)=\sum_{i\in S} d_i$, where $d_i=\sum_{j=1}^n w_{ij}$ is the degree of the vertex $i$.

%the classical Lov\'asz extension
%has its own limits that it deals with one-set form  functions corresponding to a graph $2$-cut problem.
%It can not directly work for general graph $k$-cut problem, even when $k$ is given that greater than $3$, for example,  the following $3$-cut problem, that the dual Cheeger cut problem.
\begin{definition}[dual Cheeger cut \cite{Trevisan2012,BauerJost2013}]
\rm %\label{def:Lovasz}
The dual Cheeger problem  is devoted to solving
\begin{equation}
\label{eq:dual-Cheeger-constant}
 h^+(G)= \max_{S_1\cap S_2=\varnothing, S_1\cup S_2\neq \varnothing}\frac{2|E(S_1, S_2)|}{\vol(S_1\cup S_2)},
 \end{equation}
and we call $h^+(G)$ the dual Cheeger constant.
\end{definition}

%The Lov\'asz extension aims at the equivalence
%between discrete combination optimization and continuous function
%optimization

To say the least, before we discuss Question \ref{question:graphcut},
the following specific question needs to be solved at the first place.

\begin{question}\label{question:dualcc}

Is there an explicit and equivalent continuous optimization for a graph $3$-cut problem like the dual Cheeger cut  \eqref{eq:dual-Cheeger-constant}?
\end{question}

In this work, we propose a set-pair Lov\'asz extension which not only provides a complete answer to Question \ref{question:dualcc} (even works for a series of graph $3$-cut problems),
but also enlarges the feasible region of resulting equivalent continuous optimization problems from the first quadrant $\mathbb{R}_+^n\setminus\{\vec 0\}$ (see Theorem \ref{th:Lovasz}) to the entire space $\mathbb{R}^n\setminus\{\vec 0\}$ (see Theorem \ref{th:set-pair-Lovasz}) for graph $2$-cut problems like the Cheeger cut (see Theorem \ref{th:cheeger_full}).
This enlarged feasible region may have some advantages in designing solution algorithms.

%graph $2$-cut problems.
%, such as the max 3-cut, anti-Cheeger cut, etc.,   than the original Lov\'asz extension,

\begin{definition}[set-pair Lov\'asz extension] %\cite{ChangShaoZhang2016-maxcut}
 \label{def:Lovasz-pair}
Let $V=\{1,\ldots,n\}$. For $\vec x\in \mathbb{R}^n$,  let $\sigma:V\cup\{0\}\to V\cup\{0\}$ be a bijection such that $|x_{\sigma(1)}|\le |x_{\sigma(2)}| \le \cdots\le |x_{\sigma(n)}|$ and $\sigma(0)=0$, where $x_0:=0$. One defines the sets
\begin{equation}
V_{\sigma(i)}^\pm:=\{j\in V:\pm x_j> |x_{\sigma(i)}|\},\;\;\;\; i=0,1,\ldots,n-1.
\end{equation}
Let
\begin{equation}
\power_2(V)=\{(A,B):A,B\subset V\text{ with }A\cap B=\varnothing\}.
\end{equation}
Given $f:\power_2(V)\to [0,+\infty)$, the set-pair {Lov\'asz extension} of $f$ is a mapping from $\mathbb{R}^n$ to $\mathbb{R}$ defined by
\begin{equation}\label{eq:fL}
f^{L}(\vec x)=\sum_{i=0}^{n-1}(|x_{\sigma(i+1)}|-|x_{\sigma(i)}|)f(V_{\sigma(i)}^+,V_{\sigma(i)}^-).
\end{equation}
\end{definition}

\begin{theorem}
\label{th:set-pair-Lovasz}
Assume that $f,g:\power_2(V)\to [0,+\infty)$ are two set-pair functions with $g(A,B)>0$ whenever $A\cup B\ne\varnothing$. Then there hold both
\begin{equation}\label{eq:setpair-Lovasz-min}
\min\limits_{(A,B)\in \power_2(V)\setminus\{(\varnothing,\varnothing)\}}\frac{f(A,B)}{g(A,B)}=\min\limits_{\vec   x\ne\vec  0}\frac{f^L( \vec  x)}{g^L( \vec  x)},
\end{equation}
and
\begin{equation}\label{eq:setpair-Lovasz-max}
\max\limits_{(A,B)\in \power_2(V)\setminus\{(\varnothing,\varnothing)\}}\frac{f(A,B)}{g(A,B)}=\max\limits_{\vec   x\ne\vec  0}\frac{f^L( \vec  x)}{g^L( \vec  x)}.
\end{equation}
\end{theorem}

Theorem \ref{th:set-pair-Lovasz} and its applications listed below show a natural answer to Question \ref{question:dualcc}.
%where  $\varnothing:=\emptyset$ in $\power_2(V)$.

\begin{theorem}
\label{th:dualCheeger}%\cite{ChangShaoZhang2016-maxcut}
\begin{equation}\label{eq:dualCheeger-I(x)}
1-h^+(G)=\min\limits_{\vec x\ne \vec 0}\frac{I^+(\vec x)}{\| \vec  x\| },
\end{equation}
where
\begin{align}
\| \vec  x\| &= \sum_{i=1}^n d_i|x_i|,  \\
I^+(\vec x) &=\sum_{i< j}w_{ij}|x_i+x_j|. \label{eq:I+(x)}
\end{align}
\end{theorem}

\begin{definition}[max 3-cut \cite{FriezeJerrum1997}] %\cite{ChangShaoZhang2016-maxcut}]
 %\label{def:Lovasz}
 The max 3-cut problem is to determine a graph 3-cut by solving
\begin{equation} \label{eq:max3}
h_{\max,3}(G) = \max_{A,B,C}\frac{2(|E(A,B)|+|E(B,C)|+|E(C,A)|)}{\vol(V)},
\end{equation}
and the associate $(A,B,C)$ is called a max 3-cut, where the subsets $A,B,C$ satisfy $A\cap B=B\cap C= C\cap A=\varnothing$ and $A\cup B\cup C=V$.
\end{definition}

\begin{theorem}
\label{th:max3cut}%\cite{ChangShaoZhang2016-maxcut}
\begin{equation}\label{eq:max3cut-continuous}
h_{\max,3}(G)=\max\limits_{\vec x\ne\vec0}\frac{I(\vec x)+\hat{I}(\vec x) }{\vol(V)\norm{\vec x}},
\end{equation}
where
\begin{align}
\label{eq:|I|(x)}
I(\vec x)&=\sum_{i< j}w_{ij}|x_i-x_j|,\\
\hat{I}(\vec x)&=\sum_{i<j}w_{ij}\left||x_i|-|x_j|\right|,\\
\label{eq:||x||infty}
\|\vec x\|_\infty&=\max\{|x_1|,\ldots,|x_n|\}.
\end{align}
\end{theorem}

\begin{definition}[ratio max 3-cut I] %\cite{ChangShaoZhang2016-maxcut}]
\rm %\label{def:Lovasz}
 The first ratio max 3-cut problem is to determine a graph 3-cut by solving
\begin{equation} \label{eq:max3I33}
h_{\max,3,I}(G) = \max_{A,B,C}\frac{2(|E(A,B)|+|E(B,C)|+|E(C,A)|)}{\vol(A)+\vol(B)},
\end{equation}
where $A\cap B=B\cap C= C\cap A=\varnothing$ and $A\cup B\cup C=V$.
\end{definition}

\begin{theorem}\label{th:max3cut-ratio1}%\cite{ChangShaoZhang2016-maxcut}
\begin{equation*}
1-h_{\max,3,I}(G)=\min\limits_{\vec x\ne\vec0}\frac{ I^+(\vec x)- 2\hat{I}(\vec x) }{\|\vec x\|}.
\end{equation*}
\end{theorem}

\begin{definition}[ratio max 3-cut II] %\cite{ChangShaoZhang2016-maxcut}]
\rm %\label{def:Lovasz}
 The second ratio max 3-cut problem is to determine a graph 3-cut by solving
\begin{equation} \label{eq:max3II}
h_{\max,3,II}(G) = \max_{A,B,C}\frac{2(|E(A,B)|+|E(B,C)|+|E(C,A)|)}{\max\{\vol(A\cup B),\vol(C)\}},
\end{equation}
where $A\cap B=B\cap C= C\cap A=\varnothing$ and $A\cup B\cup C=V$.
\end{definition}

\begin{theorem}\label{th:max3cut-II}%\cite{ChangShaoZhang2016-maxcut}
\begin{equation}\label{eq:tmp2}
h_{\max,3,II}(G)=\max\limits_{\vec x\ne\vec0}\frac{2I(\vec x)-\|\vec x\|+I^+(\vec x)}{\vol(V)\|\vec x\|_\infty-\min\limits_{\alpha\in \mathbb{R}}\sum_{i=1}^nd_i\left||x_i|-\alpha\right|}.
\end{equation}
\end{theorem}

In order to give a complete answer to Question \ref{question:graphcut},
we propose an isomorphism to translate a $k$-cut ($k>3$) problem to a $3$-cut one on a graph of larger size, and then 
still utilize Theorem \ref{th:set-pair-Lovasz} to derive the corresponding continuous problem.
On the other hand, Theorem \ref{th:set-pair-Lovasz} also works for graph $2$-cut problems,
although it produces a different form from that by the original Lova\'sz extension,
during which Lemma \ref{lem:} plays a key role and translates a graph $2$-cut with symmetric form into a graph $3$-cut.
The main difference lies in the feasible region.

%we should figure out how to encode `$k$ parts' into a continuous form. Fortunately, we find , and therefore, the continuous form of graph $k$-cut problem can be obtained by Theorem \ref{th:set-pair-Lovasz} directly.

%Three typical examples, including the famous maxcut, Cheeger cut and anti-Cheeger cut problems, are presented.

\begin{definition}[maxcut \cite{Karp1972}] %\cite{ChangShaoZhang2016-maxcut}]
\rm %\label{def:Lovasz}
The maxcut problem is to determine a graph cut by solving
\begin{equation} \label{eq:maxcut}
h_{\max}(G) = \max_{S\subset V}\frac{2|\partial S|}{\vol(V)}.
\end{equation}
\end{definition}
\begin{definition}[Cheeger cut \cite{BuhlerRangapuramSetzerHein2013}]
The Cheeger problem is to determine a graph cut by solving
\[
h(G)=\min\limits_{S\subset V,S\not\in\{\varnothing,V\}} \frac{|\partial S|}{\min\{\vol(S),\vol(S^c)\}}.
\]
\end{definition}
\begin{definition}[anti-Cheeger cut \cite{Xu2016}]  %\cite{ChangShaoZhang2016-maxcut}]
\rm %\label{def:Lovasz}
The anti-Cheeger constant $\ah(G)$ is defined as
\begin{equation} \label{eq:antiCheeger}
\ah(G)=\max_{S\subset V}\frac{|\partial S|}{\max\{\vol(S),\vol(S^c)\}}.
\end{equation}
\end{definition}

%Since one has no general answer for Question \ref{question:graphcut}, there comes a special question for the maxcut problem.

 %So, there arises a special question for the maxcut problem which is a case of Question \ref{question:graphcut}:

%However, for the famous maxcut problem, the Lov\'asz extension does't work.
%\begin{theorem}\rm\label{th:Cheeger}%\cite{ChangShaoZhang2016-maxcut}
%
%\end{theorem}
%through direct calculations:

The original Lov\'asz extension \eqref{eq:foL} (\videpost) yields the following equivalent continuous optimization problems: 
\begin{align}
\label{eq:maxcut-old} h_{\max}(G)& = \max_{\vec x\in \mathbb{R}^n_+\setminus\{\vec 0\}}\frac{I(\vec x)}{\vol(V)\max\limits_i x_i},\\
h(G)&=\inf_{\vec x\text{ nonconstant in } \mathbb{R}^n_+}\; \sup_{c\in \mathbb{R}}\frac{I(\vec x)}{\sum_{i=1}^n d_i|x_i-c|},\\
\ah(G)&=\max_{\vec x\in \mathbb{R}^n_+\setminus\{\vec 0\}}\frac{I(\vec x)}{2\vol(V)\max\limits_i x_i- \min\limits_{c\in \mathbb{R}}\sum_{i=1}^n d_i|x_i-c|}. \label{eq:anti-old}
\end{align}
In contrast, the proposed set-pair Lov\'asz extension is capable of enlarging the feasible region from 
the first quadrant $\mathbb{R}_+^n\setminus\{\vec 0\}$ in Eqs.~\eqref{eq:maxcut-old}-\eqref{eq:anti-old} to the entire space $\mathbb{R}^n\setminus\{\vec 0\}$ in Eqs.~\eqref{eq:maxcut-continuous}-\eqref{eq:antiCheeger-I(x)}.

%find the equivalent continuous optimization of maxcut problem. % , which cannot be solved by Theorem \ref{th:Lovasz} based the Lov\'asz extension.

%is a effective generalization of

\begin{theorem}
\label{th:maxcut}%\cite{ChangShaoZhang2016-maxcut}
\begin{equation}\label{eq:maxcut-continuous}
h_{\max}(G)=\max\limits_{\vec x\ne\vec0}\frac{I(\vec x)}{\vol(V)\norm{\vec x}}.
\end{equation}
\end{theorem}

\begin{theorem}
\label{th:cheeger_full}%\cite{ChangShaoZhang2016-maxcut}
\begin{equation}
h(G)=\inf_{\vec x \text{ nonconstant}} \sup_{c\in \mathbb{R}}\frac{I(\vec x)}{\sum_{i=1}^n d_i|x_i-c|}.
\end{equation}
\end{theorem}

\begin{theorem}
\label{th:anti-Cheeger}
\begin{equation} \label{eq:antiCheeger-I(x)}
\ah(G)=\max_{\vec x\ne \vec0}\frac{I(\vec x)}{2\vol(V)\|\vec x\|_\infty - \min\limits_{\alpha\in \mathbb{R}}\|\vec x-\alpha \vec 1\| }.
\end{equation}
\end{theorem}

\begin{remark}\rm
Comparing \eqref{eq:max3cut-continuous} to \eqref{eq:maxcut-continuous}, the continuous objective function for max 3-cut happens to only add a nonnegative term $\hat{I}(\vec x)$ to the numerator. Such slight formal discrepancy may imply some deep connections between maxcut and max 3-cut which deserves more efforts to explore.
\end{remark}

%Here it should be noted that the definitions of nodal domain of an eigenvector corresponding to $c_k$, $c_k^\infty$ and $c_k^+$ are  different.

\tikzstyle{startstop} = [rectangle, rounded corners, minimum width=1cm, minimum height=1cm,text centered, draw=black, fill=red!30]
\tikzstyle{io1} = [rectangle, trapezium left angle=80, trapezium right angle=100, minimum width=1cm, minimum height=1cm, text centered, draw=black, fill=blue!30]
\tikzstyle{io2} = [rectangle, trapezium left angle=80, trapezium right angle=100, minimum width=1cm, minimum height=1cm, text centered, draw=black, fill=yellow!10]
\tikzstyle{process} = [rectangle, minimum width=1cm, minimum height=1cm, text centered, text width=2.2cm, draw=black, fill=orange!0]
\tikzstyle{process1} = [rectangle, minimum width=1cm, minimum height=1cm, text centered,
text width=2.2cm,draw=black, fill=green!0]
\tikzstyle{process2} = [rectangle, minimum width=1cm, minimum height=1cm, text centered,
text width=2.2cm, draw=black, fill=yellow!0]
\tikzstyle{decision} = [diamond, minimum width=1cm, minimum height=1cm, text centered, draw=black, fill=green!0]
\tikzstyle{arrow} = [thick,->,>=stealth]

In a word, we may summarize the research line for the graph cut problems into the picture:
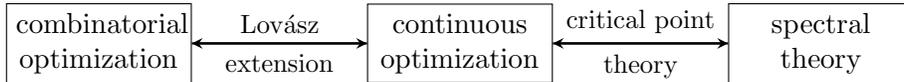
\begin{center}
\begin{tikzpicture}[node distance=4.8cm]
\node (1) [process] {combinatorial optimization};
\node (2) [process1, right of=1] {continuous optimization};
\node (3) [process2, right of=2] {spectral theory};
\draw [arrow](1) --node[anchor=south]{\small Lov\'asz}  (2);
\draw [arrow](2) --node[anchor=north]{\small extension}  (1);
\draw [arrow](2) --node[anchor=south]{\small critical point}  (3);
\draw [arrow](3) --node[anchor=north]{\small theory}  (2);
\end{tikzpicture}
\end{center}

%As a direct consequence of the equivalent optimization problem \ref{eq:maxcut-continuous}, a simple continuous maxcut algorithm, which converges in finite steps and produces approximated cuts of comparable quality to the best known ones on G-set, has been proposed.

The Lov\'asz extension and its variants provide a systematic way to find explicit equivalent continuous optimization problems for discrete and combinatorial optimization ones.
On the practical side, new possibilities immediately open up for designing continuous optimization algorithms for combinatorial problems. Several preliminary attempts have been tried \cite{HeinBuhler2010,Trevisan2012,ChangShaoZhang2015,ChangShaoZhang2016-maxcut,ChangShaoZhangZhang2018-a} and more efforts are highly called for algorithm research.  On the other hand,
explicit objective functions of resulting continuous optimizations provide the possibility to develop and enrich the spectral graph theory \cite{Chung1997} as already did for the Cheeger cut \cite{Chang2015,ChangShaoZhang2015,ChangShaoZhang2017}
and the dual Cheeger cut \cite{ChangShaoZhang2016-maxcut}.

%Besides, we have studied the dual Cheeger problem and its application to maxcut in a recent work %series of works
%\cite{ChangShaoZhang2016-maxcut}. %ChangShaoZhangZhang2018-cpt

% systemetic way

%With the help of the equivalence between discrete combination optimization and continuous function optimization, we transform  the Cheeger cut problem and dual Cheeger problem respectively into the continuous optimizations of $I^-$ and $I^+$.

%%%% organization

The rest of the paper is organized as follows.
Section \ref{sec:lovasz} collects basic properties of the Lov\'asz extension including
both continuity and convexity, and shows that the set-pair Lov\'asz extension
may be superior over the original one. Such superiority is further demonstrated in
Section \ref{sec:app} by applying it into typical graph $k$-cut problems.

%including graph $3$-cut, dual Cheeger cut, anti-Cheeger cut, maxcut and max $k$-cut problems.

\bigskip

{\bf Acknowledgements.}
This research was supported by the National Natural Science Foundation of China (Nos.~11822102, 11421101).
SS is partially supported by Beijing Academy of Artificial Intelligence (BAAI).
DZ is partially supported by the Project funded by China Postdoctoral Science Foundation (No. BX201700009).

\section{Set-pair Lov\'asz extension}
\label{sec:lovasz}

%The Lov\'asz extension,  a basic tool in combinatorics \cite{Lovasz1983, Bach2013}, aims at the establishment of the equivalence between discrete optimization and continuous optimization.

\begin{definition}[Lov\'asz extension \cite{Lovasz1983}]
\rm \label{def:originLovasz}
Let $V=\{1,\ldots,n\}\subset \mathbb{N}$.
For $\vec x\in \mathbb{R}^n$, let $\sigma:V\cup\{0\}\to V\cup\{0\}$ be a bijection such that $ x_{\sigma(1)}\le x_{\sigma(2)} \le \cdots\le x_{\sigma(n)}$ and $\sigma(0)=0$, where $x_0:=0$. One defines the sets
$$V_{\sigma(i)}:=\{j\in V: x_{j}> x_{\sigma(i)}\},\;\;\;\; i=1,\ldots,n-1,\; V_0=V.$$
Let 
\begin{equation}
\power(V)=\{A:A\subset V\}.
\end{equation} 
Given $f:\power(V)\to [0,+\infty)$, the {Lov\'asz extension} of $f$ is a mapping from $\mathbb{R}^n$ to $\mathbb{R}$ defined by
\begin{equation}\label{eq:foL}
f_o^{L}(\vec x)=\sum_{i=0}^{n-1}(x_{\sigma(i+1)}-x_{\sigma(i)})f(V_{\sigma(i)}).
\end{equation}
\end{definition}

\begin{theorem}[\cite{BuhlerRangapuramSetzerHein2013}, Theorem 1]
\label{th:Lovasz}
\rm
Assume that $f,g:\power (V)\to [0,+\infty)$ are two functions with $g(A)>0$ whenever $A\ne\varnothing$, then there hold both
\begin{equation}\label{eq:Lovasz-min}
\min\limits_{A\in \power(V)\setminus\{\varnothing\}}\frac{f(A)}{g(A)} =\min\limits_{\vec   x\in \mathbb{R}_+^n\setminus\{\vec 0\}}\frac{f_o^L( \vec  x)}{g_o^L( \vec  x)},
\end{equation}
and
\begin{equation}\label{eq:Lovasz-max}
\max\limits_{A\in \power(V)\setminus\{\varnothing\}}\frac{f(A)}{g(A)} =\max\limits_{\vec   x\in \mathbb{R}_+^n\setminus\{\vec 0\}}\frac{f_o^L( \vec  x)}{g_o^L( \vec  x)}.
\end{equation}
\end{theorem}

%In this section, we include two applications of Theorem~\ref{mainthm} and Theorem~\ref{mainthm2}.
%Such dependence is also menstifioed 
%This phenomenon can be also obtained from the integral form (see Proposition \ref{def:Lovasz-inte} and the 1st comparison of original and set-pair Lov\'asz extensions at the end of this section).

\begin{remark}\rm
The proof of Theorem \ref{th:Lovasz} (see \cite{BuhlerRangapuramSetzerHein2013}) heavily depends on the non-negativity of the terms $x_{\sigma(1)}f(V_{\sigma(1)})$ and $(x_{\sigma(i+1)}-x_{\sigma(i)})f(V_{\sigma(i)})$ in the summation form \eqref{eq:foL}, $i=1,\ldots,n-1$, thereby indicating 
that one needs the constraint $x_{\sigma(1)}\ge0$, i.e., %$\min_i x_i\ge0$, and thus one uses at least the restriction
$\vec x\in \mathbb{R}_+^n$.
The integral form \eqref{eq:Lovasz-inte} in Proposition \ref{def:Lovasz-inte} also manifests clearly such dependence 
through the last term. Indeed, the minor change of Eq.~\eqref{eq:maxcut-old}:
\begin{align*}
-\infty=\inf_{\vec x\ne \vec 0}\frac{I(\vec x)}{\vol(V)\max\limits_i x_i}<&\min_{S\subset V}\frac{2|\partial S|}{\vol(V)} \\\
\le& \max_{S\subset V}\frac{2|\partial S|}{\vol(V)} < \sup_{\vec x\ne \vec 0}\frac{I(\vec x)}{\vol(V)\max\limits_i x_i}=+\infty,
\end{align*}
shows an example in which Theorem \ref{th:Lovasz} fails if we naively replace $\mathbb{R}_+^n\setminus\{\vec 0\}$ by $\mathbb{R}^n\setminus\{\vec 0\}$ in Eqs.~\eqref{eq:Lovasz-min} and \eqref{eq:Lovasz-max}.
Fortunately, as the fruitful results and discussions in this section, we can enlarge the feasible region $\mathbb{R}_+^n\setminus\{\vec 0\}$ to $\mathbb{R}^n\setminus\{\vec 0\}$ using the proposed set-pair analog of Lov\'asz extension.
\end{remark}

\begin{prop}[\cite{Bach2013}, Definition 3.1]
\label{def:Lovasz-inte}
\begin{equation}\label{eq:Lovasz-inte}
f_o^{L}(\vec x)=\int_{\min_{1\le i\le n}x_i}^{\max_{1\le i\le n}x_i} f(V_t)d t+f(V)\min_{1\le i\le n}x_i,
\end{equation}
where $V_t(\vec x)=\{i\in V:  x_i>t\}$.
\end{prop}

%in the original Lov\'asz extension,
%???????????????????????enlarges the feasible region of resulting equivalent continuous optimization problems from a half space $\mathbb{R}_+^n\setminus\{\vec 0\}$ (see Theorem \ref{th:Lovasz}) to the entire space $\mathbb{R}^n\setminus\{\vec 0\}$

\begin{definition}
\rm
A set-function $f:\power(V)\to \mathbb{R}$ is symmetric if $f(A)=f(A^c)$ for any subset $A\subset V$. A set-pair-function $f:\power_2(V)\to \mathbb{R}$ is symmetric if  $f(A,B)=f(B,A)$ for any $(A,B)\in \power_2(V)$.
\end{definition}

\begin{prop}[\cite{Bach2013}, Proposition 3.1]\rm
\label{def:rrrr}
For $f_o^L(\vec x)$ by the Lov\'asz extension, we have
\begin{enumerate}[(a)]
\item $f_o^L(\vec x+\alpha \vec 1)=f_o^L(\vec x)+\alpha f(V)$ for any $\alpha\in \mathbb{R}$.
\item  $f_o^L(\vec x)$ is one-homogeneous.
\item $(f+g)_o^L=f_o^L+g_o^L$, $(\lambda f)_o^L=\lambda f_o^L$, $\forall \lambda\ge 0$.
\item $f_o^L(\vec x)$ is even if and only if $f$ is symmetric.
\end{enumerate}
\end{prop}
%Let $V=\{1,2,\cdots,n\}$. 
For $A\subset V$,  $\vec 1_A$ is the characteristic function of $A$.
For the set-pair case, we denote
\begin{equation}
\vec 1_{A,B}=\vec 1_A-\vec 1_B,\;\; \forall\,(A,B)\in \power_2(V).
\end{equation}
Accordingly, for any set-pair function $f:\power_2(V)\to [0,+\infty)$,  the following fact can be readily verified by Definition \ref{def:Lovasz-pair}
\begin{equation}\label{eq:fact}
f^L(\vec 1_{A,B}) = f(A,B), \;\;  \forall\, (A,B)\in\power_2(V)\setminus\{(\varnothing,\varnothing)\}.
\end{equation}
Particularly, the above equality is always true for any $(A,B)\in\power_2(V)$ if $f(\varnothing,\varnothing)=0$.

Similarly, we can derive an integral form of the set-pair Lov\'asz extension.
\begin{prop}\label{def:Lovasz2}
\begin{equation}\label{eq:Lovasz2int}
f^{L}(\vec x)=\int_0^{\|x\|_\infty} f(V_t^+(\vec x),V_t^-(\vec x))d t,
\end{equation}
where $V_t^{\pm}(\vec x)=\{i\in V: \pm x_i>t\}$.
\end{prop}

\begin{proof}
Let $\sigma$ be a permutation defined in Definition \ref{def:Lovasz-pair}. It is easy to check that  if $|x_{\sigma(i)}|\leq t< |x_{\sigma(i+1)}|$ then
$$V_t^{\pm}(\vec x)=\{i\in V: \pm x_i>t\}=V_{\sigma(i)}^\pm.$$
Therefore,
\begin{align*}
f^{L}(\vec x)&=\sum_{i=0}^{n-1}(|x_{\sigma(i+1)}|-|x_{\sigma(i)}|)f(V_{\sigma(i)}^+,V_{\sigma(i)}^-)\\
&=\sum_{i=0}^{n-1}\int_{|x_{\sigma(i)}|}^{|x_{\sigma(i+1)}|}f(V_{t}^+,V_{t}^-)dt\\
&=\int_0^{\|x\|_\infty} f(V_t^+,V_t^-)dt.
\end{align*}
\end{proof}

For simplicity, we denote by $(A,B)\subset (C,D)$ if $A\subset C$ and $B\subset D$ for $(A,B), (C,D)\in \power_2(V)$.

\begin{prop}\label{def:Lovasz1}
\begin{equation}\label{eq:Lovasz2sum}
f^{L}(\vec x)=\sum_{i=1}^p\lambda_if(V_i^+,V_i^-)
\end{equation}
whenever $ (V_p^+,V_p^-)\subset\cdots\subset (V_0^+,V_0^-)$ ($p\in\mathbb{N}^+$) is a chain  satisfying $V_0^+\cap V_0^-=\varnothing$, $\sum_{i=0}^p \lambda_i \vec 1_{V_i^+,V_i^-}=\vec x$ and $\sum_{i=0}^p \lambda_i=\|\vec x\|_\infty$, $\lambda_i\ge 0$.
\end{prop}

\begin{proof}
Setting $t_i=\sum_{j=0}^{i-1}\lambda_j$, $i=0,\ldots,p$, $t_{p+1}=\norm{\vec x}$. Now we verify  %that for any $0\le i\le p$,
\begin{equation}
\lambda_if(V_i^+,V_i^-)=\int_{t_{i-1}}^{t_i}f(V_t^+(\vec x),V_t^-(\vec x))dt, \quad i=0,\ldots,p,
\end{equation}
where $V^\pm_t(\vec x)=\{ i\in V\,|\, \pm x_i>t\}$. It obviously holds for the case of $\lambda_i=0$, so we can assume that $\lambda_i\ne 0$. Since
\begin{equation}\label{eq:lovaszcondition1}
\sum_{i=0}^p \lambda_i \vec 1_{V_i^+,V_i^-}=\vec x
\end{equation}
and
\begin{equation*}
(V_p^+,V_p^-)\subset\cdots\subset (V_0^+,V_0^-),%~\mbox{and }V_0^+\cap V_0^-=\varnothing.
\end{equation*}
we can assume that $j\in V_i^\pm$ and thus $j\notin V_i^\mp$ for any $0\le i \le p$. Consider the $j$-th component of \eqref{eq:lovaszcondition1} on both sides,
\begin{equation}\label{eq:lambdaixj}
\pm\sum\limits_{ V_i^\pm\ni j}\lambda_i =\pm\sum\limits_{i=0}^{p_j}\lambda_i=\pm t_{p_j+1}=x_j,
\end{equation}
where $p_j$ is the largest integer such that $j\in V_{p_j}^\pm$. Then
%\begin{equation}\label{eq:lovaszcondition2}
%\begin{cases} j\in V_i^\pm,&\mbox{ if }i\leq p_j,\\
%j\notin V_i^\pm,&\mbox{ if }i> p_j,
%\end{cases}
%\end{equation}
%which is equivalent to
%\begin{equation}
%\begin{cases} j\in V_i^\pm,&\mbox{ if }t_i\leq \abs{x_j},\\
%j\notin V_i^\pm,&\mbox{ if }t_i> \abs{x_j},
%\end{cases}
%\end{equation}
%because $\{t_i\}$ is an increasing sequence. So there hold
$$V_t^\pm(\vec x) = V_i^\pm,\;\; \forall\,t\in[t_{i-1},t_i),$$
and
\begin{equation*}
\int_{t_{i-1}}^{t_i}f(V_t^+(\vec x),V_t^-(\vec x))dt=\int_{t_{i-1}}^{t_i}f(V_i^+,V_i^-)dt=\lambda_if(V_i^+,V_i^-).
\end{equation*}
Therefore
\begin{align*}
\sum_{i=0}^p\lambda_if(V_i^+,V_i^-)&=\sum_{i=0}^p \int_{t_{i-1}}^{t_i}f(V_t^+(\vec x),V_t^-(\vec x))dt\\
&=\int_{0}^{t_p}f(V_t^+(\vec x),V_t^-(\vec x))dt=f^L(\vec x).
\end{align*}

In particular, let $p=n-1$, $V_i^\pm =V_{\sigma(i)}^\pm$ and
$\lambda_i = |x_{\sigma(i+1)}|-|x_{\sigma(i)}|$, $i=0,1,\ldots,n-1$,
then \eqref{eq:Lovasz2sum} returns to \eqref{eq:fL}.
\end{proof}

\begin{remark}\rm
A more detailed checking of the proof of Proposition \ref{def:Lovasz1} shows that,
for given $\vec x\ne\vec0$, if we assume every $\lambda_i>0$,  then the chain $ (V_p^+,V_p^-)\subset\cdots\subset (V_0^+,V_0^-)$ ($p\in\mathbb{N}^+$) and $\{\lambda_i\}_{i=0}^p$ in Proposition \ref{def:Lovasz1}
%\eqref{eq:Lovasz2sum}  satisfying $V_0^+\cap V_0^-=\varnothing$, $\sum_{i=0}^p \lambda_i \vec 1_{V_i^+,V_i^-}=\vec x$ and $\sum_{i=0}^p \lambda_i=\|\vec x\|_\infty$
 are uniquely determined by $\vec x$ and thus independent of $f$. That is, if there are two chains $ (V_p^+,V_p^-)\subset\cdots\subset (V_0^+,V_0^-)$ ($p\in\mathbb{N}^+$), and  $ (\widetilde{V}_q^+,\widetilde{V}_q^-)\subset\cdots\subset (\widetilde{V}_0^+,\widetilde{V}_0^-)$ ($q\in\mathbb{N}^+$), as well as two sequences of positive numbers $\{\lambda_i\}_{i=0}^p$ and $\{ \widetilde{\lambda}_i\}_{i=0}^q$, such that  $V_0^+\cap V_0^-=\widetilde{V}_0^+\cap \widetilde{V}_0^-=\varnothing$, $\sum_{i=0}^p \lambda_i \vec 1_{V_i^+,V_i^-}=\sum_{i=0}^q \widetilde{\lambda}_i \vec 1_{\widetilde{V}_i^+,\widetilde{V}_i^-}=\vec x$ and $\sum_{i=0}^p \lambda_i=\sum_{i=0}^q  \widetilde{\lambda}_i=\|\vec x\|_\infty$, then $q=p$, $(\widetilde{V}_i^+,\widetilde{V}_i^-)=(V_i^+,V_i^-)$ and  $ \widetilde{\lambda}_i =\lambda_i $, $i=0,\ldots,p$.
\end{remark}

%\begin{prop}
%For given $\vec x\ne\vec0$, if we assume every $\lambda_i>0$,  then the chain $ (V_p^+,V_p^-)\subset\cdots\subset (V_0^+,V_0^-)$ ($p\in\mathbb{N}^+$) and $\{\lambda_i\}_{i=0}^p$ in Proposition \ref{def:Lovasz1}
%%\eqref{eq:Lovasz2sum}  satisfying $V_0^+\cap V_0^-=\varnothing$, $\sum_{i=0}^p \lambda_i \vec 1_{V_i^+,V_i^-}=\vec x$ and $\sum_{i=0}^p \lambda_i=\|\vec x\|_\infty$
% are uniquely determined.
%\end{prop}
%\begin{proof}
% That is, if there are two chains $ (V_p^+,V_p^-)\subset\cdots\subset (V_0^+,V_0^-)$ ($p\in\mathbb{N}^+$), and  $ (\widetilde{V}_q^+,\widetilde{V}_q^-)\subset\cdots\subset (\widetilde{V}_0^+,\widetilde{V}_0^-)$ ($q\in\mathbb{N}^+$), as well as two sequences of positive numbers $\{\lambda_i\}_{i=0}^p$ and $\{ \widetilde{\lambda}_i\}_{i=0}^q$, such that  $V_0^+\cap V_0^-=\widetilde{V}_0^+\cap \widetilde{V}_0^-=\varnothing$, $\sum_{i=0}^p \lambda_i \vec 1_{V_i^+,V_i^-}=\sum_{i=0}^q \widetilde{\lambda}_i \vec 1_{\widetilde{V}_i^+,\widetilde{V}_i^-}=\vec x$ and $\sum_{i=0}^p \lambda_i=\sum_{i=0}^q  \widetilde{\lambda}_i=\|\vec x\|_\infty$, then $q=p$, $(\widetilde{V}_i^+,\widetilde{V}_i^-)=(V_i^+,V_i^-)$ and  $ \widetilde{\lambda}_i =\lambda_i $, $i=0,\ldots,p$.
%\end{proof}

Propositions \ref{def:Lovasz2} and \ref{def:Lovasz1} provides repectively the integral (continuous) form and chain (combinatorial) form of the set-pair Lov\'asz extension, both of which are very helpful.

The set-pair version of Proposition \ref{def:rrrr} reads as follows:
\begin{prop}\label{pro:pro}
For $f^L(\vec x)$ by the set-pair Lov\'asz extension, we have
\begin{enumerate}[(a)]
\item $f^L(\vec x+ \alpha\,\sign(\vec x))=f^L(\vec x)+\alpha f(V_0^+,V_0^-)$ for any $\alpha\ge 0$.
\item\label{pro:pro-c}  $f^L(\vec x)$ is one-homogeneous.
\item $(f+g)^L=f^L+g^L$, $(\lambda f)^L=\lambda f^L$, $\forall\,\lambda\ge 0$.
\item $f^L(\vec x)$ is even if and only if $f$ is symmetric.
\end{enumerate}
\end{prop}

\begin{proof}
We will give the proof in turn.
\begin{enumerate}[(a)]
\item Let $\tilde{\vec x}=\vec x +\alpha\,\sign(\vec x)$. Then $\norm{\tilde{\vec x}}=\norm{\vec x}+\alpha$ and
\[
\begin{cases}
V_{t}^\pm(\tilde{\vec x})=V_{t-\alpha}^\pm(\vec x),&\mbox{ if }t\geq\alpha,\\
V_t^\pm(\tilde{\vec x})=V_0^\pm(\vec x),&\mbox{ if }t\in[0,\alpha).
\end{cases}
\]
%$V_{t+\alpha}^\pm(\tilde{\vec x})=V_t^\pm(\vec x)$ if $t\ge 0$; $V_t^\pm(\tilde{\vec x})=V^\pm$, if $t\in [0,\alpha)$.
According to Proposition \ref{def:Lovasz2}, we have
\begin{align*}
f^L(\tilde{\vec x})&=\int_0^{\|\tilde{\vec x}\|_\infty} f(V_t^+(\tilde{\vec x}),V_t^-(\tilde{\vec x}))d t\\\
&=\int_0^\alpha f(V^+_0,V^-_0)dt+\int_\alpha^{\|x\|_\infty+\alpha} f(V_t^+(\tilde{\vec x}),V_t^-(\tilde{\vec x}))d t\\
&=\alpha f(V^+_0,V^-_0)+\int_0^{\|x\|_\infty} f(V_t^+({\vec x}),V_t^-({\vec x}))dt\\
&=\alpha f(V^+_0,V^-_0)+f^L(\vec x).
\end{align*}
\item For any $ \lambda>0$, we have
\begin{align*}
f^L(\lambda \vec x)&=\int_0^{\lambda\|x\|_\infty} f(V_t^+(\lambda\vec x),V_t^-(\lambda\vec x))d t\\
&=\int_0^{\|x\|_\infty} \lambda f(V_{s}^+(\vec x),V_{s}^-(\vec x))d s=\lambda f^L(\vec x).
\end{align*}

\item It can be obtained directly by the linearity of integral operators.

\item On one hand, if $f^L(\vec x)$ is even, then for any $A,B\in \power_2(V)$, there holds
\begin{align*}
f(A,B)=f^L(\vec 1_{A,B})=f^L(-\vec 1_{A,B})=f^L(\vec 1_{B,A})=f(B,A),
\end{align*}
due to \eqref{eq:fact}. Hence $f(A,B)$ is symmetric.\\
On the other hand, if $f(A,B)$ is symmetric, then using Proposition \ref{def:Lovasz2} leads to
\begin{align*}
f^{L}(\vec x)&=\int_0^{\|x\|_\infty} f(V_t^+(\vec x),V_t^-(\vec x))d t\\
&=\int_0^{\|x\|_\infty} f(V_t^-(\vec x),V_t^+(\vec x))d t\\
&=\int_0^{\|x\|_\infty} f(V_t^+(-\vec x),V_t^-(-\vec x))d t =f^{L}(-\vec x),
\end{align*}
i.e., $f^L(\vec x)$ is even.
\end{enumerate}
\end{proof}

Now we are in the position to give the proof of Theorem \ref{th:set-pair-Lovasz}.

\begin{proof}[Proof of Theorem \ref{th:set-pair-Lovasz}]
On one hand, for any $(A,B)\in \power_2(V)\setminus\{(\varnothing,\varnothing)\}$,
we have $f(A,B)=f^L(\vec 1_{A,B})$ and $g(A,B)=g^L(\vec 1_{A,B})$ due to \eqref{eq:fact}, and then
\begin{equation}\label{eq:set-pair-Lovasz1}
\min\limits_{(A,B)\in \power_2(V)\setminus\{(\varnothing,\varnothing)\}}\frac{f(A,B)}{g(A,B)}=\min\limits_{(A,B)\in \power_2(V)\setminus\{(\varnothing,\varnothing)\}}\frac{f^L(\vec 1_{A,B})}{g^L(\vec 1_{A,B})}\ge \inf\limits_{\vec x\neq \vec 0} \frac{f
^L(\vec x)}{g^L(\vec x)}.
\end{equation}
On the other hand, for any $ \vec x\ne \vec 0$, we have
$$\frac{f^{L}(\vec x)}{g^{L}(\vec x)}=\frac{\sum_{i=0}^{n-1}(|x_{\sigma(i+1)}|-|x_{\sigma(i)}|)f(V_{\sigma(i)}^+,V_{\sigma(i)}^-)}{\sum_{i=0}^{n-1}(|x_{\sigma(i+1)}|-|x_{\sigma(i)}|)g(V_{\sigma(i)}^+,V_{\sigma(i)}^-)}.$$
Let $(C,D)\in\set{(V_{\sigma(i)}^+,V_{\sigma(i)}^-)|0\le i\le n-1}$ such that
$$
\frac{f(C,D)}{g(C,D)}=\min\limits_{0\leq i \leq n-1}\frac{f(V_{\sigma(i)}^+,V_{\sigma(i)}^-)}{g(V_{\sigma(i)}^+,V_{\sigma(i)}^-)},
$$
and thus
$$
\Pi_i := g(V_{\sigma(i)}^+,V_{\sigma(i)}^-) \left(\frac{f(V_{\sigma(i)}^+,V_{\sigma(i)}^-)}{g(V_{\sigma(i)}^+,V_{\sigma(i)}^-)}
-\frac{f(C,D)}{g(C,D)}\right)\geq 0
$$
holds for any $0\le i\le n-1$. Accordingly, we have
\begin{equation*}
\frac{f^{L}(\vec x)}{g^{L}(\vec x)}-\frac{f(C,D)}{g(C,D)}
=\frac{\sum_{i=0}^{n-1}(|x_{\sigma(i+1)}|-|x_{\sigma(i)}|)\Pi_i}{\sum_{i=0}^{n-1}(|x_{\sigma(i+1)}|-|x_{\sigma(i)}|)
g(V_{\sigma(i)}^+,V_{\sigma(i)}^-)}
\ge 0,
\end{equation*}
which directly implies
\begin{equation}\label{eq:set-pair-Lovasz2}
\min\limits_{(A,B)\in \power_2(V)\setminus\{(\varnothing,\varnothing)\}}\frac{f(A,B)}{g(A,B)}=\frac{f^L(\vec 1_{C,D})}{g^L(\vec 1_{C,D})}\leq \inf\limits_{\vec x\neq \vec 0} \frac{f
^L(\vec x)}{g^L(\vec x)}.
\end{equation}
Combining \eqref{eq:set-pair-Lovasz1} and \eqref{eq:set-pair-Lovasz2} finally yields
\begin{equation*}
\min\limits_{(A,B)\in \power_2(V)\setminus\{(\varnothing,\varnothing)\} }\frac{f(A,B)}{g(A,B)}=\min\limits_{\vec   x\ne\vec  0}\frac{f^L( \vec  x)}{g^L( \vec  x)}.
\end{equation*}
The proof for the maximum problem \eqref{eq:setpair-Lovasz-max} is similar and thus skipped.
\end{proof}
A similar deduction to that for Theorem \ref{th:set-pair-Lovasz} leads to

\begin{prop}\label{pro2}
Assume that $f,g:\power_2(V)\to [0,+\infty)$ are two set-pair functions satisfying $f(\varnothing,V)=f(V,\varnothing)=0$ and $g(A,B)>0$ whenever $(A,B)\notin\{(\varnothing,\varnothing),(\varnothing,V),(V,\varnothing)\}$, then there hold both
\begin{equation}\label{eq:0-to-nonconstant}
\min\limits_{(A,B)\in \power_2(V)\setminus\{(\varnothing,\varnothing),(\varnothing,V),(V,\varnothing)\} }\frac{f(A,B)}{g(A,B)}=\min\limits_{\vec x\text{ nonconstant}}\frac{f^L( \vec  x)}{g^L( \vec  x)},
\end{equation}
and
\begin{equation}\label{eq:0-to-nonconstant-max}
\max\limits_{(A,B)\in \power_2(V)\setminus\{(\varnothing,\varnothing),(\varnothing,V),(V,\varnothing)\} }\frac{f(A,B)}{g(A,B)}=\max\limits_{\vec x\text{ nonconstant}}\frac{f^L( \vec  x)}{g^L( \vec  x)}.
\end{equation}
\end{prop}

%\begin{proof}
%
%\end{proof}

Next, we study the  continuity of $f^L$. %We assume $f(\varnothing,\varnothing)=0$.

\begin{theorem}\label{pro:Local-Lip}
$f^L$ is a Lipschitz continuous piecewise linear function.
\end{theorem}

\begin{proof}
For a mapping $m:\{1,2,\ldots,n\}\to \{-1,1\}$ and a permutation $\sigma$ of $\{1,2,\ldots,n\}$, one defines a closed convex cone as follows
$$\triangle_{m,\sigma}:=\left\{\vec x\in\mathbb{R}^n:\abs{x_{\sigma(1)}}\leq\cdots\leq \abs{x_{\sigma(n)}}\mbox{ with } x_{\sigma(i)}m(i)\ge 0\right\},$$
and it can be readily seen that $\mathbb{R}^n=\bigcup_{m,\sigma}\triangle_{m,\sigma}$.

It suffices to prove that $f^L$ is linear and Lipschitz continuous with a Lipschitz constant $2\max_{(A,B)\in\power_2(V)}f(A,B)$ on each $\triangle_{m,\sigma}$. In fact, for given $m$ and $\sigma$ and any $\vec x\in \triangle_{m,\sigma}$, we have
\begin{align*}
f^L(\vec x)=&\sum_{i=0}^{n-1} (m(i+1)x_{\sigma(i+1)}-m(i)x_{\sigma(i)})f(V_{x_{\sigma(i)}}^+,V_{x_{\sigma(i)}}^-)\\
=&\sum_{i=1}^{n-1}x_{\sigma(i)}m(i)(f(V_{x_{\sigma(i-1)}}^+,V_{x_{\sigma(i-1)}}^-)-f(V_{x_{\sigma(i)}}^+,V_{x_{\sigma(i)}}^-))\\
&+x_{\sigma(n)}m(n)f(V_{x_{\sigma(n-1)}}^+,V_{x_{\sigma(n-1)}}^-).
\end{align*}
Since $m(i)$ and $f(V_{x_{\sigma(i)}}^+,V_{x_{\sigma(i)}}^-)$ are constants for given $m$ and $\sigma$, $f^L$ is linear on $\triangle_{m,\sigma}$. Moreover, for any $\vec x,\vec y \in \triangle_{m,\sigma}$,
\begin{align*}
\abs{f^L(\vec x)-f^L(\vec y)}
\le & \sum_{i=1}^{n-1}\abs{x_{\sigma(i)}-y_{\sigma(i)}}\abs{f(V_{x_{\sigma(i-1)}}^+,V_{x_{\sigma(i-1)}}^-)-f(V_{x_{\sigma(i)}}^+,V_{x_{\sigma(i)}}^-)}\\
&+\abs{x_{\sigma(n)}-y_{\sigma(n)}}f(V_{x_{\sigma(n-1)}}^+,V_{x_{\sigma(n-1)}}^-)\\
\le& 2\max_{(A,B)\in\power_2(V)}f(A,B)\Vert \vec x-\vec y\Vert_1.
\end{align*}
\end{proof}

The concept of submodular function was  introduced  by Lov\'asz to characterize the convexity of its Lov\'asz extension \cite{Lovasz1983}.

\begin{definition}[submodular function \cite{Lovasz1983}]\rm
A set-function $f:\power(V)\to \mathbb{R}$ is submodular if and only if,
for all subsets $A,B\subset V$,
$$f(A)+f(B)\ge f(A\cup B)+f(A\cap B).$$
\end{definition}

\begin{theorem}[\cite{Lovasz1983}, Proposition 4.1]\rm
\label{th:Lovasz-convex}
$f^L_o$ is convex if and only if $f$ is submodular.
\end{theorem}

Theorem \ref{th:Lovasz-convex} inspires us to consider the set-pair form of submodular function.
A kind of set-pair submodular function was proposed \cite{BercziFrank2008}. %, but it does't involve the corresponding  Lov\'asz extension.

\begin{definition}[set-pair submodular function \cite{BercziFrank2008}] \label{defn:pair-submodular}\rm
Let
$$
\power_2'(V)=\{(X_I,X_O):X_I\subset X_O\subset V\}.
$$
A function $p:\power_2'(V)\to \mathbb{R}$ is submodular if
 \begin{equation}\label{eq:p-submodular}
p(X_I,X_O)+p(Y_I,Y_O)\ge p(X_I\cap Y_I,X_O\cap Y_O)+p(X_I\cup Y_I,X_O\cup Y_O)
\end{equation}
for any $(X_I,X_O),(Y_I,Y_O)\in \power_2'(V)$.

By taking $p(X_I,X_O)=f(X_I,X_O\setminus X_I)$ and $f(A,B)=p(A,A\cup B)$, we can transform from $f:\power_2(V)\to \mathbb{R}$ to $p:\power_2'(V)\to \mathbb{R}$ and vice versa.
Such $f$ and $p$ are said to be {\sl equivalent}.
\end{definition}

Now we show necessary and sufficient conditions for the convexity of $f^L$.

\begin{theorem}\label{th:pair-Lovasz-convex}
Let $f:\power_2(V)\to [0,+\infty)$ is a set-pair functions satisfying $f(\varnothing,\varnothing)=0$.
Then $f^L$ is convex if and only if $\forall\,(A,B),(C,D)\in \power_2(V)$ %$f$ satisfies the inequality
\begin{equation}\label{eq:2-submodular}
f(A,B)+f(C,D)\ge f((A\cup C)\setminus (B\cup D),(B\cup D)\setminus(A\cup C))+f(A\cap C,B\cap D);
\end{equation}
if and only if $\forall\,(X_I,X_O),(Y_I,Y_O)\in \power_2'(V)$ the equivalent function $p$
satisifies
 \begin{equation}\label{eq:2-submodular-p}
p(X_I,X_O)+p(Y_I,Y_O)\ge p(X_I\cap Y_I,X_O\cap Y_O\setminus Z)
+
p((X_I\cup Y_I)\setminus Z,(X_O\cup Y_O)\setminus Z),
\end{equation}
where
$Z=( X_O\cap Y_I\setminus X_I)\cup (Y_O\cap X_I\setminus Y_I)$.
\end{theorem}

Three lemmas below are needed in proving Theorem \ref{th:pair-Lovasz-convex}.

\begin{lemma}\label{lemma:strictmod1}
For $\vec x\in\mathbb{R}^n$, $N\in\mathbb{N}^+$, and $N > 2\norm{\vec x}$,
let
\begin{equation}\label{eq:fhat}
\hat{f}_N(\vec x)=
\min\left\{\sum_{(A,B)\in \power_2(V)}\lambda_{A,B}f(A,B)\left|\begin{array}{l}
 \sum\lambda_{A,B}\vec 1_{A,B}=\vec x,  \\
 \sum\lambda_{A,B}\leq N ,\\
 \lambda_{A,B}\ge 0.
\end{array}
\right.\right\}.
\end{equation}
Then $\hat{f}_N(\vec x)$ is convex.
\end{lemma}
\begin{proof}
The fact that $\hat{f}_N(\vec x)$ is well defined emerges from Proposition \ref{def:Lovasz1}.
Given $\vec x,\vec y\in \mathbb{R}^n$, from \eqref{eq:fhat}, we deduce
$$\hat{f}_N(\vec x)=\sum_{(A,B)\in \power_2(V)}\alpha_{A,B}f(A,B)$$
holds for some $\alpha_{A,B}\ge 0$ with $\sum\alpha_{A,B}\vec 1_{A,B}=\vec x$. Similarly,
$$\hat{f}_N(\vec y)=\sum_{(A,B)\in \power_2(V)}\beta_{A,B}f(A,B)$$
holds for some $\beta_{A,B}\ge 0$ with $\sum\beta_{A,B}\vec 1_{A,B}=\vec y$.

Let $\lambda_{A,B}=t \alpha_{A,B}+(1-t)\beta_{A,B}$ with $t\in[0,1]$. Immediately, we have
$$\vec z := t\vec x+(1-t)\vec y = \sum_{(A,B)\in \power_2(V)} \lambda_{A,B}\vec 1_{A,B}$$ with $\sum\lambda_{A,B}\leq N$ and $\lambda_{A,B}\ge 0$,
and then
$$\hat{f}_N(\vec z)\leq \sum_{(A,B)\in \power_2(V)} \lambda_{A,B}f(A,B)= t \hat{f}_N(\vec x)+(1-t)\hat{f}_N(\vec  y).$$
\end{proof}

%$\vec z=t\vec x+(1-t)\vec y$ where $t\in[0,1]$.

\begin{definition}\rm
A set-pair function $f:\power_2(V)\to [0,+\infty)$ is said to be strictly submodular if,
the inequality \eqref{eq:2-submodular} holds. 
Moreover, the equality holds if and only if $(A,B)\subset (C,D)$ or $(A,B)\supset (C,D)$.
\end{definition}

\begin{lemma}\label{lemma:strictmod2}
If $f$ is strictly submodular, then $\hat{f}_N(\vec x) = f^L(\vec x)$ for $N>c\norm{\vec x}$ with $c>1$.
\end{lemma}

\begin{proof}
Given $\vec x\in \mathbb{R}^n$, according to Lemma \ref{lemma:strictmod1},
there exist $\lambda_{A, B}\ge 0, \, \forall (A, B)\in P_2(V)$ with $\sum
\lambda_{A, B} \vec 1_{A,B}= \vec x$ and $\sum \lambda_{A, B}\le N$ such that
$$
\hat{f}_N(\vec x)=\sum_{(A,B)\in \power_2(V)}\lambda_{A,B}f(A,B).
$$
No loss of generality, we can assume $\lambda_{\varnothing,\varnothing}=0$.

% and thus we may assume that  $(A,B)\ne (\varnothing,\varnothing)$.

We claim:  if $\lambda_{A,B}\ge\lambda_{C,D}>0$, then
either $(A,B)\subset(C,D)$ or $(C,D)\subset (A,B)$. Suppose the contrary and let
\begin{equation*}
\left\{\begin{split}
&\lambda'_{A,B}=\lambda_{A,B}-\lambda_{C,D},\\
&\lambda'_{C,D}=0,\\
&\lambda'_{A',B'}=\lambda_{A',B'}+\lambda_{C,D},\\
&\lambda'_{C',D'}=\lambda_{C',D'}+\lambda_{C,D},\\
&\lambda'_{E,F}=\lambda_{E,F},\forall\,(E,F)\in\power_2(V)\setminus\set{(A,B),(C,D),(A',B'),(C',D')},
\end{split}
\right.
\end{equation*}
where
$$
A'=(A\cup C)\setminus (B\cup D), B'=(B\cup D)\setminus(A\cup C), C'=A\cap C,
D'=B\cap D.
$$
Then it can be easily verified that
$$\sum\limits_{(P,Q)\in\power_2(V)}\lambda'_{P,Q}=\sum\limits_{(P,Q)\in\power_2(V)}\lambda_{P,Q}\;
\text{ and }\sum\limits_{(P,Q)\in\power_2(V)}\lambda'_{P,Q}\vec 1_{P,Q}=\vec x.$$
Direct calculation shows
\begin{align*}
&\sum\limits_{(P,Q)\in\power_2(V)}\lambda'_{P,Q}f(P,Q)-\sum\limits_{(P,Q)\in\power_2(V)}\lambda_{P,Q}f(P,Q)\\
=\,&\sum\limits_{(P,Q)\in\power_2(V)}(\lambda'_{P,Q}-\lambda_{P,Q})f(P,Q)\\
=\,&\lambda_{C,D}(-f(A,B)-f(C,D)+f(A',B')+f(C',D'))<0,
\end{align*}
provided the strict submodularity of $f$.
This contradicts the minimality of $\hat{f}(\vec x)$.
According to the mathematical induction we obtain $ (\varnothing,\varnothing)\ne(V^+_p,V^-_p)\subset \cdots\subset(V^+_0,V^-_0)$ with
$$
\hat{f}_N(\vec x)=\sum_{i=0}^p\lambda_{V^+_i,V^-_i}f(V^+_i,V^-_i).% = f^L(\vec x)
$$
Moreover,  we have   $\sum_{i=0}^p\lambda_{V^+_i,V^-_i}=\|\vec x\|_\infty$ via
\eqref{eq:lambdaixj}. After Proposition \ref{def:Lovasz1}, $\hat{f}_N(\vec x) = f^L(\vec x)$.
\end{proof}
\begin{lemma}\label{lemma:strictmod3}
The function
$$g(A,B):=\sqrt{\abs{A}+\abs{B}}$$
 is strictly submodular.
\end{lemma}
\begin{proof}
Given $(A,B),(C,D)\in \power_2(V)$, let $A'=A\setminus D$, $B'=B\setminus C$, $C'=C\setminus B$ and $D'=D\setminus A$. Then
\begin{align*}
&g((A\cup C)\setminus (B\cup D),(B\cup D)\setminus(A\cup C))+g(A\cap C,B\cap D)\\
%=\,&g(A'\cup C',B'\cup D')+g(A'\cap C',B' \cap D')\\
=\,&\sqrt{\abs{A'\cup C'}+\abs{B'\cup D'}}+\sqrt{\abs{A'\cap C'}+\abs{B'\cap D'}}\\
\le \,& \sqrt{\abs{A'}+\abs{B'}}+\sqrt{\abs{C'}+\abs{D'}}\\
\le \,& \sqrt{\abs{A}+\abs{B}}+\sqrt{\abs{C}+\abs{D}}=g(A,B)+g(C,D),
\end{align*}
where the first inequality holds since the function $\sqrt{t}$ is strictly convex.
Meanwhile, we can easily see that the equality holds if and only if $(A,B)\subset (C,D)$ or $(C,D)\subset (A,B)$. The proof is thus completed.
\end{proof}

\begin{proof}[Proof of Theorem \ref{th:pair-Lovasz-convex}]
%We prove that $f^L$ is convex if \eqref{eq:2-submodular} holds.
%On one hand, suppose that $f$ is submodular.
Suppose that $f$ satisfies \eqref{eq:2-submodular}.
For any $ \alpha>0$,  $f+\alpha g$ is strictly submodular according to Lemma \ref{lemma:strictmod3}.
Thus, by Lemma \ref{lemma:strictmod2}, we have
\begin{equation}\label{eq:fl1}
f^L+\alpha g^L=(f+\alpha g)^L=\widehat{f+\alpha g}\ge \hat{f}.
\end{equation}

Given $\vec x\in \mathbb{R}^n$,  set $\hat{f}=\hat{f}_N$ for fixed $N>2\norm{\vec x}$. Hence,   \eqref{eq:fhat} leads to
$$\hat{f}(\vec x)=\sum_{(A,B)\in \power_2(V)}\lambda_{A,B}f(A,B)$$
 for some $\lambda_{A,B}\ge 0$ with $\sum\lambda_{A,B}\vec 1_{A,B}=\vec x$, $\sum\lambda_{A,B}<N$. Then
\begin{align}
f^L(\vec x)+\alpha g^L(\vec x) = \widehat{f+\alpha g}(\vec x)&\le  \sum_{(A,B)\in \power_2(V)}\lambda_{A,B}(f(A,B)+\alpha g(A,B)) \nonumber\\
&= \hat{f}(\vec x)+\alpha \sum_{(A,B)\in \power_2(V)}\lambda_{A,B} g(A,B) \nonumber \\
&\le  \hat{f}(\vec x)+\alpha \sum_{(A,B)\in \power_2(V)}\lambda_{A,B}\sqrt{n}\nonumber\\
&\le   \hat{f}(\vec x)+\alpha N \sqrt{n}.\label{eq:fl2}
\end{align}
Letting $\alpha\to 0$ in \eqref{eq:fl1} and \eqref{eq:fl2} yields
$$f^L(\vec x) = \hat{f}(\vec x),$$
and by Lemma \ref{lemma:strictmod1}, $\hat{f}(\vec x)$ is convex, so is $f^L(\vec x)$.

On the other hand, if $f^L(\vec x)$ is convex, then
\begin{align*}
&f((A\cup C)\setminus (B\cup D),(B\cup D)\setminus(A\cup C))+f(A\cap C,B\cap D)\\
=\;&f^L(\vec 1_{(A\cup C)\setminus (B\cup D),(B\cup D)\setminus(A\cup C)}+\vec 1_{A\cap C,B\cap D}) \quad [\text{Proposition \ref{def:Lovasz1}}]\\
=\;&f^L(\vec 1_{A,B}+\vec 1_{C,D}) \\
=\;& 2 f^L({(\vec 1_{A,B}+\vec 1_{C,D})}/{2}) \quad [\text{Proposition \ref{pro:pro}: \eqref{pro:pro-c}}]\\
\le\;&  f^L(\vec 1_{A,B})+f^L(\vec 1_{C,D}) \quad [\text{convexity}]\\
=\;& f(A,B)+f(C,D), \quad [\text{Equation \eqref{eq:fact}}]
\end{align*}
where we have used $\vec 1_{A,B}+\vec 1_{C,D} = \vec 1_{(A\cup C)\setminus (B\cup D),(B\cup D)\setminus(A\cup C)}+\vec 1_{A\cap C,B\cap D}$ in the third line. Thus, \eqref{eq:2-submodular} is true.

Finally, let
$X_I=A,~X_O=A\cup B,~Y_I=C,~Y_O=C\cup D$, then
$$Z=( X_O\cap Y_I\setminus X_I)\cup (Y_O\cap X_I\setminus Y_I)=(B\cap C)\cup(D\cap A).$$
By taking $f(A,B)=p(A,A\cup B)$, we can translate inequality \eqref{eq:2-submodular-p} into
\begin{align*}
&f(A,B) +f(C,D) =  p(X_I,X_O)+p(Y_I,Y_O)\\
\ge\;& p(X_I\cap Y_I,X_O\cap Y_O\setminus Z)
+
p((X_I\cup Y_I)\setminus Z,(X_O\cup Y_O)\setminus Z)\\
=\;&p(A\cap C, (A\cap C)\cup (B\cap D))\\
+&p((A\cup C)\setminus (B\cup D),((A\cup C)\setminus (B\cup D))\cup((B\cup D)\setminus(A\cup C)))\\
=\;&f(A\cap C,B\cap D)+f((A\cup C)\setminus (B\cup D),(B\cup D)\setminus(A\cup C)),
\end{align*}
which means \eqref{eq:2-submodular} and \eqref{eq:2-submodular-p} are equivalent.

Hence $f^L$ is convex if and only if either \eqref{eq:2-submodular}  or \eqref{eq:2-submodular-p} holds.
\end{proof}

Comparing  \eqref{eq:2-submodular-p} of Theorem \ref{th:pair-Lovasz-convex} to \eqref{eq:p-submodular} of Definition \ref{defn:pair-submodular}, we are able to deduce that the submodularity introduced in Definition \ref{defn:pair-submodular} for a set-pair function fails to ensure
an extension of  the equivalence stated in Theorem \ref{th:Lovasz-convex} between convexity and submodularity for $f_o^L$ into $f^L$ (i.e., Definition \ref{defn:pair-submodular}  is neither necessary nor sufficient  for $f^L$ to be convex), whereas \eqref{eq:2-submodular} succeeds. In such sense, we might call the set-pair function satisfying \eqref{eq:2-submodular} or \eqref{eq:2-submodular-p} to be submodular.
However, both \eqref{eq:2-submodular} and \eqref{eq:2-submodular-p} are not so easy-looking that we give a concise necessary condition for $f^L$ to be convex.

\begin{definition}\rm
A set-pair-function $f:\power_2(V)\to \mathbb{R}$ is {\sl partially submodular} if and only if it is submodular for each component, i.e.,
\begin{equation*}\label{eq:pair-Lovasz-f-1-1}
f(A,B)+f(A,D)\ge f(A,B\cup D)+f(A,B\cap D),
\end{equation*}
\begin{equation*}\label{eq:pair-Lovasz-f-1-2}
f(A,B)+f(C,B)\ge f(A\cup C,B)+f(A\cap C,B),
\end{equation*}
 for all subsets $A,B,C,D\subset V$ with $A\cap B=A\cap D=C\cap B=\varnothing$.
\end{definition}

%$$f(A,B)=p(A,A\cup B)$$

%$$p(X_I,X_O)+p(Y_I,Y_O)\ge p(X_I\cup Y_I,X_O\cup Y_O)+p(X_I\cap Y_I,X_O\cap Y_O)$$

%$X_I=A$, $X_O=A\cup B$, $Y_I=C$, $Y_O=C\cup D$
%
%$(A\cup C)\setminus (B\cup D)=(X_I\cup Y_I)\setminus (X_O\setminus X_I \cup Y_O\setminus Y_I)=((X_I\cup Y_I)\setminus (X_O\cap Y_O)) \cup (X_I\cap Y_I)$
%
%$(A\cup C)\triangle (B\cup D)= (X_O\cup Y_O)\setminus ((Y_I\cap (X_O\setminus X_I))\cup (X_I\cap (Y_O\setminus Y_I)))$
%
%\begin{equation}\label{eq:pair-Lovasz-p-1}
%p(X_I,X_O)+p(Y_I,Y_O)\ge p(X_I\cup Y_I,X_O\cup Y_O)+p(X_I\cap Y_I,X_O\cap Y_O)
%\end{equation}
%
%$(A\cap C)\cup (B\cap D)=(X_O\cap Y_O)\setminus ((Y_I\cap (X_O\setminus X_I))\cup (X_I\cap (Y_O\setminus Y_I)))$
%
%$(Y_I\cap (X_O\setminus X_I))\cup (X_I\cap (Y_O\setminus Y_I))=(Y_I\cap X_O)\cup (X_I\cap Y_O)\setminus (X_I\cap Y_I)$
%
%
%$$p(X_I,X_O)+p(Y_I,Y_O)\ge p(X_I\cap Y_I,X_O\cap Y_O)+p(X_I\cup Y_I,X_O\cup Y_O)$$
%
%$$
%p(X_I,X_O)+p(Y_I,Y_O)\ge p(X_I\cap Y_I,(X_O\cap Y_O)\setminus (( X_O\cap Y_I\setminus X_I)\cup (Y_O\cap X_I\setminus Y_I)))
%$$
%$$+
%p(((X_I\cup Y_I)\setminus (X_O\cap Y_O)) \cup (X_I\cap Y_I),(X_O\cup Y_O)\setminus (( X_O\cap Y_I\setminus X_I)\cup (Y_O\cap X_I\setminus Y_I)))
%$$

\begin{corollary}\label{cor:convex-submodular}
If $f^L$ is convex, then $f$ must be partially submodular.
\end{corollary}

\begin{proof}
If $f^L$ is convex, then $f$ must satisfy \eqref{eq:2-submodular}.
Setting $C=A$ and $D=B$ respectively in \eqref{eq:2-submodular},
we can find that $f$ is partially submodular.
\end{proof}

%\begin{remark}
Similar to Corollary \ref{cor:convex-submodular}, if $p$ is  submodular, then its equivalent function $f$ must be  partially submodular.
%\end{remark}

Finally, we compare the set-pair Lov\'asz extension to the original one:

\begin{enumerate}
\item  In contrast to the succinct integral form \eqref{eq:Lovasz2int} of $f^L$, the integral form \eqref{eq:Lovasz-inte} of $f_o^L$ has an extra remainder term.

%which  in Example \ref{example:Lovasz}. %, and this term makes the

\item The original Lov\'asz extension is unable directly to deal with graph $3$-cut problems such as
the dual Cheeger-cut problem, whereas the set-pair Lov\'asz extension works.

\item The characterization of the convexity of $f_o^L$ is easier than $f^L$.
\end{enumerate}

%Let $\sigma$ be a permutation defined in Definition \ref{def:originLovasz}. The original Lov\'asz extension of $f$ and $g$ can be computed as follows:

\section{Applications to graph cut}
\label{sec:app}

A straightforward application of the original Lov\'asz extension \eqref{eq:foL} into a graph $3$-cut problem such as the dual Cheeger problem is not feasible. Instead, we will show in this section that the set-pair Lov\'asz extension \eqref{eq:fL} can succeed to find an explicit and equivalent continuous optimization problem for graph $3$-cut. To be more specific,
the set-pair Lov\'asz extension of the five set-pair functions is summarized in Table~\ref{tab:L}.

%constitutes one of the main tasks of this section.
%Moreover, it
%And in the second subsection, we give a solution for graph $k$-cut problem by translating it into a $3$-cut problem on a larger graph. This allows us to use set-pair Lov\'asz extension \eqref{eq:fL} to study it. Success in giving equivalent continuous form for graph $k$-cut problem reveals that it is a promising solution.
%Furthermore,  the set-pair Lov\'asz extension can give a different form of continuous optimization problem which has a constraint in the whole space instead of only positive half space.
%The key is the following lemma which transforms equivalently the combination optimization from a set form into a set-pair form. %to such success

%And it seems that we cannot use Theorem \ref{th:set-pair-Lovasz} directly since the discrete object function of maxcut is not set-pair form.
%The core of the solution for these difficulties is the following result, which turns the discrete optimization in set form to an equivalent set-pair optimization.
%A key lemma

%For convenience, in this section we define five functions from $\power_2(V)$ to $[0,+\infty)$ and show their set-pair Lov\'asz extensions as follows.

\begin{table}
\centering
\caption{\small Set-pair Lov\'asz extension of five object functions.}
\begin{tabular}{|l|l|}
               \hline
               Object function & Set-pair Lov\'asz extension  \\
               \hline
                 $F_1(A,B)=\abs{\partial A}+\abs{\partial B}$&$F_1^L(\vec x)=I(\vec x)$\\
                 \hline
               $F_2(A,B)=\abs{E(A,B)}$&$F_2^L(\vec x)=\frac{1}{2}\|\vec x\|-\frac{1}{2}I^+(\vec x)$\\
               \hline
               $G_1(A,B)=\,\vol(V)$&$G_1^L(\vec x)=\,\vol(V)\|\vec x\|_\infty$\\
               \hline
               $G_2(A,B)=\vol(A)+\vol(B)$&$G_2^L(\vec x)=\|\vec x\|$\\
               \hline
               $G_3(A,B)=\sum\limits_{X\in \{A,B\}}\min\limits_{Y\in \{X,X^c\}} \vol(Y)$&$G_3^L(\vec x)=\min\limits_{\alpha\in \mathbb{R}}\|\abs{\vec x}-\alpha \vec 1\|$\\
               \hline
            \end{tabular}
             \label{tab:L}
\end{table}

%In Table \ref{tab:L}, $I(\vec x) = \sum_{i<j}w_{ij}|x_i-x_j|$, $I^+(\vec x) = \sum_{i<j}w_{ij}|x_i+x_j|$ and $\|\vec x\| = \sum_{i=1}^n d_i|x_i|$.

%\begin{align}
%F_1(A,B)&=\abs{\partial A}+\abs{\partial B},&F_1^L(\vec x)&=I(\vec x),\label{F1}\\
%F_2(A,B)&=\abs{E(A,B)},&F_2^L(\vec x)&=\frac{1}{2}\|\vec x\|-\frac{1}{2}I^+(\vec x),\label{F2}\\
%G_1(A,B)&=\,\vol(V),&G_1^L(\vec x)&=\,\vol(V)\|\vec x\|_\infty,\label{G1}\\
%G_2(A,B)&=\vol(A)+\vol(B),&G_2^L(\vec x)&=\|\vec x\|,\label{G2}\\
%G_3(A,B)&=\sum_{X\in \{A,B\}}\min_{Y\in \{X,X^c\}} \vol(Y),&G_3^L(\vec x)&=\min\limits_{\alpha\in \mathbb{R}}\|\abs{\vec x}-\alpha \vec 1\|,\label{G3}
%\end{align}

The first four functions in Table~\ref{tab:L} can be calculated directly according to Proposition \ref{def:Lovasz2}. In fact,
\begin{equation}\label{eq:F1}
F_1^L(\vec x) = \int_0^{\norm{\vec x}}\vert \partial V_t^+(\vec x)\vert +\vert \partial V_t^-(\vec x)\vert dt. %= \int_0^{\norm{\vec x}}F_1(V_t^+(x),V_t^-(x))dt
\end{equation}
Then substituting
\begin{align*}
\vert \partial V_t^+(\vec x)\vert &= \sum_{i< j}  w_{ij}(\chi_{x_i \le t<x_j} + \chi_{x_j\le t<x_i}),\\
\vert \partial V_t^-(\vec x)\vert &= \sum_{i< j}  w_{ij}(\chi_{x_i<-t\le x_j}+\chi_{x_j<-t\le x_i})
\end{align*}
into Eq.~\eqref{eq:F1} yields
\begin{align*}
F_1^L(\vec x)&=  \sum_{i< j}w_{ij}\int_0^{\norm{\vec x}}  \chi_{x_i\le t<x_j} + \chi_{x_j\le t<x_i} +\chi_{x_i<-t\le x_j}+\chi_{x_j<-t\le x_i} dt \nonumber\\
%&= \sum_{i< j} \int_0^{\norm{\vec x}}  1_{x_i>t>x_j}+ 1_{x_i<-t<x_j} +  1_{x_j>t>x_i} +  1_{x_j<-t<x_i} dt\\
&= \sum_{i< j}w_{ij} \int_{-\norm{\vec x}}^{\norm{\vec x}}   \chi_{x_i<t<x_j} + \chi_{x_j<t<x_i}  dt\\
&= \sum_{i< j} w_{ij}\vert x_i-x_j\vert=I(\vec x),
\end{align*}
where the integral equalities hold when '$<$' is replaced by '$\le$'.

Hereafter the endpoints of intervals in the integral form \eqref{eq:Lovasz2int} are dropped
for convenience. Thus, it gives the form of set-pair Lov\'asz extension of $F_1$ in Table \ref{tab:L}.

Applying Proposition \ref{def:Lovasz2} to $F_2$, we get
\begin{align*}
F_2^L(\vec x)&=\int_0^{\norm{\vec x}} F_2(V_t^+(\vec x),V_t^-(\vec x))dt\\
&=\int_0^{\norm{\vec x}} \abs{E(V_t^+(\vec x),V_t^-(\vec x))}dt\\
&=\int_0^{\norm{\vec x}} \sum_{i=1}^n\sum_{j=1}^n w_{ij}\chi_{x_i>t}\chi_{x_j<-t}dt\\
&= \sum_{i=1}^n\sum_{j=1}^n w_{ij}\int_0^{\norm{\vec x}}\chi_{x_i>t}\chi_{x_j<-t}dt\\
&= \frac12\sum_{i=1}^n\sum_{j=1}^n w_{ij}\min\{x_i+|x_i|,|x_j|-x_j\}\\
&= \frac14\sum_{i=1}^n\sum_{j=1}^n w_{ij}(x_i+|x_i|+|x_j|-x_j-|x_i+x_j+|x_i|-|x_j||),
\end{align*}
where we used the fact that $\min\{a,b\}= \frac12(a+b-|a-b|)$ in the last equality. Further, we can obtain
\begin{align}
\notag F_2^L(\vec x)&= \frac14\sum_{i<j} w_{ij}(2(|x_i|+|x_j|)-\big|x_i+x_j+|x_i|-|x_j|\big|-\big|x_i+x_j-|x_i|+|x_j|\big|)\\
\label{eq:F2L2}
&= \frac14\sum_{i<j} w_{ij}(2|x_i|+2|x_j|-2\max\{|x_i+x_j|,\big||x_i|-|x_j|\big|\})\\
&= \frac{1}{2}\|\vec x\|-\frac{1}{2}\sum_{i< j}\abs{x_i+x_j}=\frac{1}{2}\|\vec x\|-\frac{1}{2}I^+(\vec x),\notag
\end{align}
where Eq.~\eqref{eq:F2L2} utilizes the fact that $2\max\{|a|,|b|\}=|a+b|+|a-b|$.

Correspondingly, the set-pair extensions of $G_1$ and $G_2$ in Table \ref{tab:L} can be verified in the following way:
\begin{align}
G_1^L(\vec x) &=  \int_0^{\norm{\vec x}}G_1(V_t^+(\vec x),V_t^-(\vec x))dt  = \vol(V)\norm{\vec x}, \notag\\
G_2^L(\vec x)&=\int_0^{\norm{\vec x}} G_2(V_t^+(\vec x),V_t^-(\vec x))dt \nonumber \\
&=\int_0^{\norm{\vec x}} \vol(V_t^+(\vec x))+\vol(V_t^-(\vec x))dt\nonumber \\
&=\int_0^{\norm{\vec x}} \sum_{i=1}^n d_i \chi_{\abs{x_i}>t}dt=\|\vec x\|.\notag
\end{align}
Now, let us focus on $G_3$. Direct calculation shows
\begin{align*}
G_3^L(\vec x) &= \int_0^{\norm{\vec x}}G_3(V_t^+(\vec x),V_t^-(\vec x))dt\\
&=\int_0^{\norm{\vec x}}\min\{\vol(V_t^+(\vec x)),\vol(V_t^+(\vec x)^c)\}+\min\{\vol(V_t^-(\vec x)),\vol(V_t^-(\vec x)^c)\}dt\\
&=\int_0^{\norm{\vec x}}\min \{\vol(V_t^+(\vec x)),\vol(V_t^+(\vec x)^c)\}+\min\{\vol(V_{-t}^+(\vec x)^c),\vol(V_{-t}^+(\vec x))\}dt\\
&=  \int_{-\norm{\vec x}}^{\norm{\vec x}}\min\{\vol(V_t^+(\vec x)),\vol(V_t^+(\vec x)^c)\} dt.
\end{align*}
Let $\sigma$ be a permutation of $\set{1,2,\ldots,n}$ such that
$x_{\sigma(1)}\leq x_{\sigma(2)}\leq\cdots\leq x_{\sigma(n)}$. Then there exists $k_0\in\set{1,2,\ldots,n}$ satisfying
\begin{equation}
\sum_{i=1}^{k_0-1} d_{\sigma(i)}<\frac{1}{2}\vol(V)\leq \sum_{i=1}^{k_0} d_{\sigma(i)}.
\end{equation}
%we call the set
%\begin{equation}
%\left\{
%\begin{aligned}
%[x_{\sigma(k_0)},x_{\sigma(k_0+1)}],\mbox{ if }\frac{1}{2}\vol(V)\leq \sum_{i=1}^{k_0} d_{\sigma(i)},\\
%\set{\sigma(k_0)},mbox{ if }\frac{1}{2}\vol(V)< \sum_{i=1}^{k_0} d_{\sigma(i)},
%\end{aligned}
%\right.
%\end{equation}
%is the median of $\vec x$, denoted by $median(\vec x)$.
Consequently, it reveals that
\begin{equation}
\min\{\vol(V_t^+(\vec x)),\vol(V_t^+(\vec x)^c)\} =
\begin{cases}
\vol(V_t^+(\vec x)), &\text{if $t<x_{\sigma(k_0)}$,}\\
\vol(V_t^+(\vec x)^c), &\text{if $t\ge x_{\sigma(k_0)}$,}
\end{cases}
\end{equation}
and
\begin{align*}
G_3^L(\vec x)&= \int_{x_{\sigma(1)}}^{x_{\sigma(k_0)}} \vol(V_t^+(x)^c) dt+\int_{x_{\sigma(k_0)}}^{x_{\sigma(n)}} \vol(V_t^+(x))dt\\
&= \sum_{i=1}^{k_0-1}(x_{\sigma(i+1)}-x_{\sigma(i)})\sum_{j=1}^{i} d_{\sigma(j)}+\sum_{i=k_0}^{n-1}(x_{\sigma(i+1)}-x_{\sigma(i)})\sum_{j=i+1}^{n} d_{\sigma(j)}\\
&=\sum_{j=1}^{k_0-1} d_{\sigma(j)}\sum_{i=j}^{k_0-1}(x_{\sigma(i+1)}-x_{\sigma(i)})+\sum_{j=k_0+1}^{n} d_{\sigma(j)}\sum_{i=k_0}^{j-1}(x_{\sigma(i+1)}-x_{\sigma(i)})\\
&=\sum_{i=1}^n d_{\sigma(i)}\vert x_{\sigma(i)}-x_{\sigma(k_0)}\vert = \Vert \vec x -x_{\sigma(k_0)}\vec 1\Vert.
\end{align*}
On the other hand, $\Vert \vec x -\alpha\vec 1\Vert$ is convex in $\alpha$ and satisfies 
\begin{equation*}
p_\alpha:=-\sum_{i=1}^n d_{\sigma(i)}\sign(x_{\sigma(i)}-\alpha)\in \partial_\alpha \Vert \vec x -\alpha\vec 1\Vert
\end{equation*}
and then
\begin{equation*}
\begin{cases}
p_\alpha\leq \sum_{j=1}^{k_0-1} d_{\sigma(j)}-\sum_{j=k_0}^{n} d_{\sigma(j)}\leq 0,&\text{ if $\alpha<x_{\sigma(k_0)}$,}\\
p_\alpha\ge \sum_{j=1}^{k_0} d_{\sigma(j)}-\sum_{j=k_0+1}^{n} d_{\sigma(j)}\ge 0,&\text{ if $\alpha>x_{\sigma(k_0)}$.}
\end{cases}
\end{equation*}
This implies that $ \Vert \vec x -\alpha\vec 1\Vert$ is decreasing with respect to $\alpha$ in $(-\infty,x_{\sigma(k_0)})$ and increasing in $(x_{\sigma(k_0)},+\infty)$. Thus, we obtain that
\begin{equation*}
x_{\sigma(k_0)}\in\argmin_{\alpha\in\mathbb{R}} \Vert \vec x -\alpha\vec 1\Vert.
\end{equation*}
Therefore,
\begin{equation}\label{G3L}
G_3^L(\vec x)=\Vert \vec x -x_{\sigma(k_0)}\vec 1\Vert=\min\limits_{\alpha\in \mathbb{R}}\|\vec x-\alpha \vec 1\|.
\end{equation}

%\begin{comment}
%$$f_1(A,B)=\abs{\partial A}+\abs{\partial B},~ g_1(A,B)=\vol(V),$$
%$$f_2(A,B)=\abs{E(A,B)},~g_2(A,B)=\vol(A)+\vol(B),$$
%and
%$$g_3(A,B)=\min\{ \vol(A),\vol(A^c)\}+\min\{ \vol(B),\vol(B^c)\}.$$
%\end{comment}

\subsection{Graph $3$-cut problems}
It is straightforward to utilize the proposed set-pair Lov\'asz extension 
to deal with the combination optimizations in a set-pair form,
for example, the  dual Cheeger and max 3-cut problems.
Now, we are in a position to answer Question \ref{question:dualcc}.

\begin{proof}[Proof of Theorem \ref{th:dualCheeger}]
Applying Theorem \ref{th:set-pair-Lovasz} in the dual Cheeger cut problem \eqref{eq:dual-Cheeger-constant} yields
\begin{equation}
h^+(G)=\max\limits_{\vec x\ne \vec 0}\frac{2F_2^L(\vec x )}{G_2^L(\vec x)}=1-\min\limits_{\vec x\ne \vec 0}\frac{I^+(\vec x)}{\| \vec  x\|},
\end{equation}
where we have used $F_2^L$ and $G_2^L$ in Table \ref{tab:L}.
\end{proof}

%\subsection{max 3-cut}

\begin{proof}[Proof of Theorem \ref{th:max3cut}]
Since $A\cap B=B\cap C= C\cap A=\varnothing$ and $A\cup B\cup C=V$, we can suppose that $A\cup B\ne \varnothing$. Then, by Theorem \ref{th:set-pair-Lovasz} and Proposition \ref{pro:pro}, we have
\begin{align*}
\frac{1}{2}h_{\max,3}(G) &= \max_{A,B,C}\frac{|E(A,B)|+|E(B,C)|+|E(C,A)|}{\vol(V)}\\
&=\max_{A,B,C}\frac{|E(A,B)|+|E(A\cup B,C)|}{\vol(V)}\\
&=\max\limits_{(A,B)\in \power_2(V)\setminus \{(\varnothing,\varnothing)\}}\frac{|E(A,B)|+|\partial (A\cup B)|}{\vol(V)}\\
&=\max\limits_{(A,B)\in \power_2(V)\setminus \{(\varnothing,\varnothing)\}}\frac{\abs{\partial A}+\abs{\partial B}-\abs{E(A,B)}}{\vol(V)} \\
&=\max\limits_{(A,B)\in \power_2(V)\setminus \{(\varnothing,\varnothing)\}}\frac{F_1(A,B)-F_2(A,B)}{G_1(A,B)}\\
&=\max\limits_{\vec   x\ne\vec  0}\frac{F_1^L( \vec  x)-F_2^L( \vec  x)}{G_1^L( \vec  x)},
\end{align*}
where $F_1^L$ and $G_1^L$ are given in Table \ref{tab:L}. Thus
%\begin{equation}
%h_{\max,3}(G)=\max_{\vec x\ne \vec0}\frac{f_1^L(\vec x)-f_2^L(\vec x)}{g_1^L(\vec x)}=\max_{\vec x\ne \vec0}\frac{I(\vec x)-\frac{1}{2}\|\vec x\|+\frac{1}{2}I^+(\vec x)}{\vol(V)\norm{\vec x}}.
%\end{equation}
\begin{equation}\label{eq:tmp1}
h_{\max,3}(G)=\max_{\vec x\ne \vec0}\frac{2I(\vec x)-\|\vec x\|+I^+(\vec x)}{\vol(V)\norm{\vec x}}.
\end{equation}
It follows from $|a-b|+|a+b|=2\max\{|a|,|b|\}=||a|-|b||+|a|+|b|$ for any $a,b\in \mathbb{R}$ that
\begin{align*}
I(\vec x)&=\sum_{i<j}w_{ij}|x_i-x_j|
\\&=\sum_{i< j} w_{ij}(|x_i|+|x_j| +||x_i|-|x_j||-|x_i+x_j|)
\\&=\sum_{i=1}^n d_i|x_i| + \sum_{i< j}w_{ij} ||x_i|-|x_j||- \sum_{i< j} w_{ij}|x_i+x_j|
\\&= \|\vec x\|+\hat{I}(\vec x)-I^+(\vec x),
\end{align*}
and thus
$$
2I(\vec x)-\|\vec x\|+I^+(\vec x)=2\hat{I}(\vec x)+\|\vec x\|-I^+(\vec x)=I(\vec x)+\hat{I}(\vec x).
$$
Finally, Eq.~\eqref{eq:tmp1} turns out to be
\begin{equation*}
h_{\max,3}(G)=\max_{\vec x\ne \vec0}\frac{I(\vec x)+\hat{I}(\vec x)}{\vol(V)\norm{\vec x}}.
\end{equation*}
\end{proof}

\begin{proof}[Proof of Theorem \ref{th:max3cut-ratio1}]

According to Theorem \ref{th:set-pair-Lovasz}, we have
%\begin{equation}
%h_{\max,3,I}(G)=\max\limits_{\vec x\ne\vec0}\frac{f_1^L(\vec x)-f_2^L(\vec x)}{g_3^L(\vec x)}=\max\limits_{\vec x\ne\vec0}\frac{I(\vec x)-\frac{1}{2}\|\vec x\|+\frac{1}{2}I^+(\vec x)}{\|\vec x\|}.
%\end{equation}
\begin{align*}
h_{\max,3,I}(G)=&\max\limits_{(A,B)\in \power_2(V)}\frac{F_1(A,B)-F_2(A,B)}{G_2(A,B)}\\
=&\max\limits_{\vec x\ne\vec0}\frac{\|\vec x\|-I^+(\vec x)+2\hat{I}(\vec x)}{\|\vec x\|}.
\end{align*}
\end{proof}

\begin{proof}[Proof of Theorem \ref{th:max3cut-II}]
Let
\begin{equation}
G(A,B)=\min\{\vol(A\cup B),\vol((A\cup B)^c)\},
\end{equation}
and $\sigma$ be a permutation of $\set{1,2,\ldots,n}$ such that
$\abs{x_{\sigma(1)}}\leq \abs{x_{\sigma(2)}}\leq\cdots\leq \abs{x_{\sigma(n)}}$.
By Definition \ref{def:Lovasz-pair},
the set-pair Lov\'asz extension of $G(A,B)$ is
\begin{align*}
G^{L}(\vec x)&=\sum_{i=0}^{n-1}(|x_{\sigma(i+1)}|-|x_{\sigma(i)}|)G(V_{\sigma(i)}^+,V_{\sigma(i)}^-)\\
&=\sum_{i=0}^{n-1}(|x_{\sigma(i+1)}|-|x_{\sigma(i)}|)\min\{\vol(V_{\sigma(i)}^+\cup V_{\sigma(i)}^-),\vol((V_{\sigma(i)}^+\cup V_{\sigma(i)}^-)^c)\}\\
&=\sum_{i=0}^{n-1}(|x_{\sigma(i+1)}|-|x_{\sigma(i)}|)G_3(V_{\sigma(i)}^+\cup V_{\sigma(i)}^-,\varnothing)\\
&=G_3^L(\abs{\vec x})=\min\limits_{\alpha\in \mathbb{R}}\|\abs{\vec x}-\alpha \vec 1\|,
\end{align*}
where $\abs{\vec x} = (|x_1|,\ldots,|x_n|)$ and Eq.~\eqref{G3L} is applied in the last line.
Finally, applying Theorem \ref{th:set-pair-Lovasz} into Eq.~\eqref{eq:max3II} leads to
\begin{align*}
h_{\max,3,II}(G)&=\max\limits_{\vec x\ne\vec0}\frac{2F_1^L(\vec x)-2F_2^L(\vec x)}{G_1^L(\vec x)-G^L(\vec x)}\\
&=\max\limits_{\vec x\ne\vec0}\frac{2I(\vec x)-\|\vec x\|+I^+(\vec x)}{\vol(V)\|\vec x\|_\infty-\min\limits_{\alpha\in \mathbb{R}}\sum_{i=1}^nd_i\left||x_i|-\alpha\right|}.
\end{align*}
\end{proof}

%\begin{definition}
%Denote
%$$
%\power_k([n]) = \{(A_1,\ldots,A_k)|A_i\cap A_j = \varnothing, A_i\subset [n]\},
%$$
%and
%$$
%H_{k+1}([n]) = \{(A_1,\ldots,A_{k+1})|A_i\cap A_j = \varnothing, \bigcup_{i=1}^{k+1} A_i= [n]\}.
%$$
%\end{definition}

\subsection{Graph $k$-cut ($k>3$) problems}\label{sec:kcut}
%\newpage
In this section, we present a preliminary attempt to a graph $k$-cut problem. The main idea is to transfer a graph $k$-cut problem to a $3$-cut one on a larger graph. To this end, let us start from
\begin{align}
\power_k([n]) &= \{(A_1,\ldots,A_k)\big|A_i\cap A_j = \varnothing, A_i\subset [n]\}, \\
\mathcal{H}_{k+1}([n]) &= \{(A_1,\ldots,A_{k+1})\big|A_i\cap A_j = \varnothing, \bigcup_{i=1}^{k+1} A_i= [n]\},
\end{align}
where $[n]=\{1,2,\ldots,n\}$. Obviously $\power_k([n])$ is equivalent to $\mathcal{H}_{k+1}([n])$, 
$\power_k([n])\simeq \mathcal{H}_{k+1}([n])$.
Then a bijection between $\power_2([ln])$ and $\mathcal{H}_{3^l}([n])$ can be obtained via
\begin{equation}
\power_2([ln])\simeq \mathcal{H}_{3}([ln])\simeq \prod_{i=1}^l \mathcal{H}_{3}([n])\simeq \mathcal{H}_{3^l}([n]).
\end{equation}

For a family $\power$ consisting of set-tuples, we use $C(\power):=\{f:\power\to \mathbb{R}\}$ to denote the collection of real valued functions on $\power$, and then have the following commutative diagrams for any $k<3^l$:
\begin{figure}[H]
%$$
%\xymatrix { \power_k([n]) ^  }
%$$
\begin{tikzpicture}
  \matrix (m) [matrix of math nodes, row sep=3em, column sep=3em]
    { \power_k([n])& \power_2([ln]) &C(\power_k([n]))&C(\power_2([ln]))\\
      \mathcal{H}_{3^l}([n]) & \prod\limits_{i=1}^l\mathcal{H}_{3}([n])&C(\mathcal{H}_{3^l}([n]) )& C(\prod\limits_{i=1}^l\mathcal{H}_{3}([n])) \\ };
  { [start chain] \chainin (m-1-2);
    \chainin (m-1-1)[join={node[above,labeled] {h}}];}
    { [start chain] \chainin (m-1-2);
    \chainin (m-2-2) [join={node[left,labeled] {h_3}}];
    \chainin (m-2-1) [join={node[above,labeled] {h_2}}];
    \chainin (m-1-1) [join={node[left,labeled] {h_1}}];
    }
      { [start chain] \chainin (m-1-3);
    \chainin (m-1-4)[join={node[above,labeled] {h^*}}];}
    { [start chain] \chainin (m-1-3);
    \chainin (m-2-3) [join={node[left,labeled] {h_1^*}}];
    \chainin (m-2-4) [join={node[above,labeled] {h_2^*}}];
    \chainin (m-1-4) [join={node[left,labeled] {h_3^*}}];
    }
\end{tikzpicture}
\caption{\small The  commutative diagram on the right is indeed the dual diagram of the left one in some sense.
 %Here $C(\power_k([n])):=\{f:\power_k([n])\to \mathbb{R}\}$ collects the real valued functions on $\power_k([n])$.
}\label{commutative}
\end{figure}
% $$
%\xymatrix{
%    \power_k([n]) \ar[r]^{h}\ar[d]_{h_1} & \power_2([ln]) \\
%   \mathcal{H}_{3^l}([n]) \ar[r]^{h_2}   &\prod\limits_{i=1}^l\mathcal{H}_{3}([n]) \ar[u]^{h_3}
%    }
%$$

 In Fig.~\ref{commutative}, $h_1$ is the natural injective mapping from $\mathcal{H}_{3^l}([n])$ to $\power_{k}([n])$ by choosing only the last $k$ parts from each element in $H_{3^l}([n])$.
Therefore, given $f\in C(\power_{k}([n]))$, there exists $F\in C(\power_2([ln]))$ and an injective mapping $h$ from $\power_2([ln])$ to $\power_{k}([n])$  such that $ F = f\circ h$.
 For convenience, the set-pair Lov\'asz extension of $F$ is again called the Lov\'asz extension of $f$.
 That is, there exist $F_1\in C(H_{3^l}([n]))$ and $F_2\in C(\prod_{i=1}^lH_{3}([n]))$ such that
\begin{equation}
F_1=f\circ h_1,\,\,
F_2(\prod_{i=1}^l(T^i_0,T^i_1,T^i_2)) = F_1((A_0,A_1,\ldots,A_{3^l-1})),
\end{equation}
where $(T^i_0,T^i_1,T^i_2)\in H_3([n])$, $i=1,\ldots, l$, and $A_j = \bigcap_{i=1}^l T^i_{a_i}$ with $(a_l\ldots a_1)_3$ being the ternary representation of $j$ for $j=0,1,\ldots,3^l-1$. In other words, there exists an injection $h_2$ from $\prod_{i=1}^lH_{3}([n])$ to $H_{3^l}([n])$ such that
\begin{equation}
F_2 = F_1\circ h_2.
\end{equation}

Finally, we can take an injection $h_3$ from $\power_2([ln])$ to $ H_{3^l}([n])$ by
$$
h(T_1,T_2) = \prod_{i=1}^l(T^i_0,T^i_1,T^i_2),
$$
where
\begin{equation}\label{eq:Tia}
T^i_a = \{t\in [n]|\,t+n(i-1)\in T_a\}, \quad T_0 = (T_1\cap T_2)^c, \quad a=0,1,2.
\end{equation}

Thus, the correspondence is well established between the function $f$ defined on $P_{k}([n])$ and the function $F$ on $P_2([ln])$ by letting $h=h_1\circ h_2\circ h_3$ and
$$
F = F_2\circ h_3 = F_1\circ h_2\circ h_3=f\circ h.
$$

Applying Theorem \ref{th:set-pair-Lovasz},
we are able to give an equivalent continuous optimization for the max $k$-cut problem.
% below.
\begin{definition}[max $k$-cut \cite{FriezeJerrum1997}]
Given a connected graph $G=(V,E)$ with $V=[n]$, the max $k$-cut problem is  to determine a graph $k$-cut by solving
\begin{equation}\label{eq:maxk}
h_k(G)=\min_{(A_1,A_2,\ldots,A_k)\in \mathcal{H}_{k}([n])}\frac{\sum_{i=1}^k|\partial A_i|}{ \sum_{i=1}^k\vol(A_i)}.
\end{equation}
\end{definition}

We can find $F,G\in C(\power_2([ln]))$ such that
\begin{equation}
F(T_1,T_2) = \sum_{j=3^l-k}^{3^l-1} |\partial A_j|,\text{ and }\,G(T_1,T_2) = \sum_{j=3^l-k}^{3^l-1}\vol(A_i),
\end{equation}
where $A_j = \bigcap_{i=1}^l T^i_{a_i}$, $T^i_{a_i}$ is given in Eq.~\eqref{eq:Tia},
and $(a_l\ldots a_1)_3$ denotes the ternary representation of $j$ for $j\in\{0,1,\ldots,3^l-1\}$.

Let us write down the functions $F^L$ and $G^L$ explicitly. In fact,
\begin{equation}\label{F0Lk}
F^L(\prod_{i=1}^l\vec x^{(i)}) = \int_0^{\|\vec x\|_\infty} F(V_t^+,V_t^-)dt
\end{equation}
is a continuous function defined on $\mathbb{R}^{nl}$.
Here $V_t^\pm = \{(i-1)n+j|\pm x^{(i)}_j>t \}$ for any $\vec x=\prod_{i=1}^l\vec x^{(i)}\in \mathbb{R}^{nl}$, and $\vec x^{(i)} = (x_1^{(i)},\ldots,x_n^{(i)})$.

\begin{prop}\label{eq:FkL}
$$
F^L(\prod_{i=1}^l\vec x^{(i)}) = \sum_{j=1}^n d_jz_j-2\sum_{i< j}^n w_{ij} \sum_{(a_l\ldots a_1)_2=3^l-k}^{3^l-1}z_{ij}^{(a_l\ldots a_1)_3},
$$
where
\begin{align*}
z_j &= \min\left\{t\ge0\Big|\,(a_l\ldots a_1)_3<3^l-k,\,\, a_i=\vec 1_{x_j^{(i)}>t}+2\vec 1_{-x_j^{(i)}>t}\right\}, \\
z_{ij}^{(a_l\ldots a_1)_3}&=\min\left\{-(-1)^{a_\alpha}x_{i'}^{(\alpha)}-| x_{j'}^{(\beta)}|\Big| a_\alpha >0, a_\beta = 0, i',j'\in\{i,j\}\right\}_+, \\
z_+ &= \max\{z,0\}.
\end{align*}
\end{prop}

\begin{proof}
A direct calculation leads to
\begin{align*}
F^L(\prod_{i=1}^l\vec x^{(i)}) =& \int_0^{\|\vec x\|}F(V_t)dt=\int_0^{\|\vec x\|}\sum_{(a_l\ldots a_1)_3=3^l-k}^{3^l-1}|\partial A_{a_l\ldots a_1}(t)|dt\\
=&\int_0^{\|\vec x\|}\sum_{(a_l\ldots a_1)_3=3^l-k}^{2^l-1}|\vol( A_{a_l\ldots a_1}(t))|-2|E(A_{a_l\ldots a_1}(t))|dt\\
=&\mathrm{I}-\mathrm{II},
\end{align*}
where $E(A)$ is the set of all edges with endpoints in $A$ and
\begin{align*}
\mathrm{I} =& \int_0^{\|\vec x\|}\sum_{(a_l\ldots a_1)_3=3^l-k}^{3^l-1}|\vol( A_{a_l\ldots a_1}(t))|dt,\\
\mathrm{II}  =&2\int_0^{\|\vec x\|}\sum_{(a_l\ldots a_1)_3=3^l-k}^{3^l-1}|E(A_{a_l\ldots a_1}(t))|dt.
\end{align*}
It is easy to check the first part
\begin{align*}
\mathrm{I} =&\int_0^{\|\vec x\|}\sum_{(a_l\ldots a_1)_3=3^l-k}^{3^l-1}\sum_{i< j}w_{ij}[\vec 1_{i\in A_{a_l\ldots a_1}(t)}+\vec 1_{j\in A_{a_l\ldots a_1}(t)}]dt\\
=&\sum_{i< j}w_{ij}\int_0^{\|\vec x\|}\sum_{(a_l\ldots a_1)_3=3^l-k}^{3^l-1}\vec 1_{i\in A_{a_l\ldots a_1}(t)}+\vec 1_{j\in A_{a_l\ldots a_1}(t)}dt\\
=&\sum_{i< j}w_{ij}\int_0^{\|\vec x\|}\vec 1_{i\in \bigcup_{(a_l\ldots a_1)_3=3^l-k}^{3^l-1}A_{a_l\ldots a_1}(t)}+\vec 1_{j\in \bigcup_{(a_l\ldots a_1)_3=3^l-k}^{3^l-1}A_{a_l\ldots a_1}(t)}dt\\
=&\sum_{i< j}w_{ij}\left(z_i+z_j\right)=\sum_{j=1}^nd_jz_j,
\end{align*}
and the second part
\begin{align*}
\mathrm{II}=&2\int_0^{\|\vec x\|}\sum_{(a_l\ldots a_1)_3=3^l-k}^{3^l-1}\sum_{i< j}w_{ij}\vec 1_{i,j\in A_{a_l\ldots a_1}(t)}dt\\
=&2\sum_{i< j}w_{ij}\sum_{(a_l\ldots a_1)_3=3^l-k}^{3^l-1}\int_0^{\|\vec x\|}\vec 1_{i,j\in A_{a_l\ldots a_1}(t)}dt\\
=&2\sum_{i< j}w_{ij}\sum_{(a_l\ldots a_1)_3=3^l-k}^{3^l-1}\int_0^{\|\vec x\|} \prod_{\substack{\alpha:a_\alpha>0,\\i'\in\{i,j\}}}\vec 1_{t\le-(-1)^{a_\alpha} x_{i'}^{(\alpha)}}\prod_{\substack{\beta:,a_\beta=0,\\j'\in\{i,j\}}}\vec 1_{|x_{j'}^{(\beta)}|<t}dt\\
=&2\sum_{i< j}w_{ij}\sum_{(a_l\ldots a_1)_3=3^l-k}^{3^l-1}z_{ij}^{(a_l\ldots a_1)_3}.\\
\end{align*}
Thus, we complete the proof.
\end{proof}
\begin{remark}
If $k=3^l-1$, then
$$
F^L(\prod_{i=1}^l\vec x^{(i)}) = \sum_{j=1}^n d_i\max_s|x_j^{(s)}|-2\sum_{i< j}^n w_{ij} \sum_{(a_l\ldots a_1)_3=1}^{3^l-1}z_{ij}^{(a_l\ldots a_1)_3}.
$$
\end{remark}

%There is no difficulty to verify the following
It can be readily verified that:
\begin{prop}\label{eq:GkL}
$$
G^L(\prod_{i=1}^l\vec x^{(i)})=\sum_{j=1}^nd_jz_j = \mathrm{I}.
$$
\end{prop}

Accordingly, we get
\begin{prop}
\begin{equation}
h_{k}= \max_{\vec x\in \mathbb{R}^{nl}\setminus K}\frac{F^L(\vec x)}{G^L(\vec x)},
\end{equation}
where $F^L$ and $G^L$ are defined in Propositions \ref{eq:FkL} and \ref{eq:GkL}, respectively, and
$K = \{\vec x| \,z_j=0,\forall j=1,2,\ldots,n\}$.

 %$K:=(G^L)^{-1}(0)$, i.e.,
%\begin{align*}
%K =& \{\vec x| \,z_j=0,\forall j=1,2,\ldots,n\}.
%%=&\{\vec x| \,(a_l\ldots a_1)_3<3^l-k,\,\, a_i=\vec 1_{x_j^{(i)}>0}+2\vec 1_{-x_j^{(i)}>0},\forall j=1,2,\ldots,n\}.
%\end{align*}
\end{prop}

%This completes the continuous form of max $k$-cut problem.
%with whose aid Theorem \ref{th:set-pair-Lovasz} displays its powerful capability.
%systematically studied

\subsection{Graph $2$-cut problems}

With the help of the following lemma, the proposed set-pair Lov\'asz extension
also works for some graph $2$-cut problems.

%The feasible region of the resulting
%equivalent continuous optimization problems is almost the whole space
%instead of only positive half space.

%The key is the following lemma which transforms equivalently the combination optimization from a set form into a set-pair form. %to such success

\begin{lemma}\label{lem:}
%Suppose $f:\power(V)\to [0,+\infty)$ and $g:\power(V)\to (0,+\infty)$ and are two symmetric functions.
Suppose $f,g:\power(V)\to [0,+\infty)$  are two symmetric functions with $g(A)>0$ for any $A\in \power(V)$.
Let $F(A,B)=f(A)+f(B)$ and $G(A,B)=g(A)+g(B)$. Then
\begin{align}
\min\limits_{A\in \power(V)}\frac{f(A)}{g(A)}&=\min\limits_{(A,B)\in \power_2(V)}\frac{F(A,B)}{G(A,B)}, \label{eq:min1}\\
\max\limits_{A\in \power(V)}\frac{f(A)}{g(A)}&=\max\limits_{(A,B)\in \power_2(V)}\frac{F(A,B)}{G(A,B)}.\label{eq:max2}
\end{align}
\end{lemma}

\begin{proof}
First we prove \eqref{eq:min1}. On one hand, let $(A_0, B_0)\in \power_2(V)$ be the minimizer of $\frac{f(A)+f(B)}{g(A)+g(B)}$. Without loss of generality, we may assume $\frac{f(A_0)}{g(A_0)}\leq\frac{f(B_0)}{g(B_0)}$. Then
$$
\frac{f(A_0)+g(B_0)}{g(A_0)+g(B_0)} -\frac{f(A_0)}{g(A_0)} =  \frac{f(B_0)g(A_0)-f(A_0)g(B_0)}{g(A_0)(g(A_0)+g(B_0))}\geq 0,
$$
which follows that
$$
\min\limits_{(A,B)\in \power_2(V)}\frac{f(A)+f(B)}{g(A)+g(B)}=\frac{f(A_0)+g(B_0)}{g(A_0)+g(B_0)}\ge\frac{f(A_0)}{g(A_0)}\ge \min\limits_{A\in \power(V)}\frac{f(A)}{g(A)}.
$$
On the other hand, let $A_1\subset V$ be the minimizer of $\frac{f(A)}{g(A)}$. Then we have
$$\min\limits_{(A,B)\in \power_2(V)}\frac{f(A)+f(B)}{g(A)+g(B)}\leq \frac{f(A_1)+f(A_1^c)}{g(A_1)+g(A_1^c)}= \frac{f(A_1)}{g(A_1)}=\min\limits_{A\in \power(V)}\frac{f(A)}{g(A)},$$
and hence,
$$
\min\limits_{A\in \power(V)}\frac{f(A)}{g(A)}=\min\limits_{(A,B)\in \power_2(V)}\frac{f(A)+f(B)}{g(A)+g(B)}.
$$

It is also true if we replace `min' by 'max', i.e., \eqref{eq:max2} holds.
\end{proof}
\begin{proof}[Proof of Theorem \ref{th:maxcut}]
The proof involves $F_1$ and $G_1$. Let
\begin{equation*}
f(A)=\abs{\partial A}\;\;\text{ and }\;\;g(A)=\frac{1}{2} \vol(V).
\end{equation*}
Since $f$ and $g$ are symmetric functions, by Lemma \ref{lem:} and Theorem \ref{th:set-pair-Lovasz}, we have
\begin{align*}
h_{\max}(G) &= \max_{S\in \power(V)}\frac{2|\partial S|}{\vol(V)}=\max_{S\in \power(V)}\frac{f(S)}{g(S)}=\max\limits_{(A,B)\in \power_2(V)}\frac{F_1(A,B)}{G_1(A,B)}
\\&=\max\limits_{(A,B)\in \power_2(V)\setminus \{(\varnothing,\varnothing)\}}\frac{F_1(A,B)}{G_1(A,B)}=\max\limits_{\vec   x\ne\vec  0}\frac{F_1^L( \vec  x)}{G_1^L( \vec  x)}.
\end{align*}
Accordingly, we have
\begin{equation*}
h_{\max}(G)=\max\limits_{\vec   x\ne\vec  0}\frac{F_1^L(\vec x)}{G_1^L(\vec x)}=\max\limits_{\vec   x\ne\vec  0}\frac{I(\vec x)}{\vol(V)\norm{\vec x}}.
\end{equation*}
\end{proof}

%Besides the maxcut problem,
%the set-pair Lov\'asz extension \eqref{eq:fL} also succeeds
%to yield the equivalent continuous optimization problems
%for the Cheeger cut and the  anti-Cheeger cut  with the help of Lemma \ref{lem:}.

%\subsection{anti-Cheeger cut}
\begin{proof}[Proof of Theorem \ref{th:cheeger_full}]
Let $f(A)=\abs{\partial A}$ and $g(A)=\min\{ \vol(A),\vol(A^c)\}$.
Since $f$, $g$ are symmetric functions, by Eq.~\eqref{eq:0-to-nonconstant}, Lemma \ref{lem:} and Proposition \ref{pro2}, we have
\begin{align*}
h(G)&=\min_{S\subset V}\frac{f(S)}{g(S)}=\min\limits_{(A,B)\in \power_2(V)\setminus\{(\varnothing,\varnothing),(\varnothing,V),(V,\varnothing)\} }\frac{F_1(A,B)}{G_2(A,B)}  \\
&=\min\limits_{\vec x \text{ nonconstant}}\frac{F_1^L( \vec  x)}{G_2^L( \vec  x)}
=\inf_{\vec x \text{ nonconstant}} \sup_{c\in \mathbb{R}}\frac{I(\vec x)}{\sum_{i=1}^n d_i|x_i-c|},
\end{align*}
where $F_1^L$ and $G_2^L$ have been presented in Table~\ref{tab:L}.
\end{proof}
\begin{proof}[Proof of Theorem \ref{th:anti-Cheeger}]
Let
\begin{equation}
f(A)=\abs{\partial A}\;\;\text{ and }\;\;g(A)=\max\{ \vol(A),\vol(A^c)\}.
\end{equation}
Since $f$, $g$ are symmetric functions,
by Lemma \ref{lem:} and Theorem \ref{th:set-pair-Lovasz},
we obtain
\begin{align*}
\ah(G)&=\max_{S\in\power(V)}\frac{f(S)}{g(S)}=
\max\limits_{(A,B)\in\power_2(V)\}}\frac{F_1(A,B)}{G(A,B)} \\
&=\max\limits_{(A,B)\in\power_2(V)\setminus\{(\varnothing,\varnothing)\}\}}\frac{F_1(A,B)}{G(A,B)} =
\max\limits_{\vec   x\ne\vec  0}\frac{F_1^L( \vec  x)}{G^L( \vec  x)},
\end{align*}
where $G(A,B) = 2G_1(A,B)-G_3(A,B)$,
and $G^L= 2G_1^L-G_3^L$. Thus,
\begin{equation*}
\ah(G)=\max_{\vec x\ne \vec0}\frac{F_1^L(\vec x)}{2G_1^L(\vec x)-G_3^L(\vec x)}=\max_{\vec x\ne \vec0}\frac{I(\vec x)}{2\vol(V)\|\vec x\|_\infty - \min\limits_{\alpha\in \mathbb{R}}\|\vec x-\alpha \vec 1\| }.
\end{equation*}
\end{proof}

%In summary, with the aid of Lemma \ref{lem:},
%the set-pair Lov\'asz extension \eqref{eq:fL} has much more potential
%than the original Lov\'asz extension \eqref{eq:foL}.

%\section{Discussion and outlook}

%\section{Spectral theory for maxcut problem}\label{s:boundary}

%\begin{proof}[Proof of Theorem \ref{th:nodal-domain} for $c_k^\infty$]\end{proof}

%\newpage

%\section{Spectral theory for dual Cheeger problem}\label{s:boundary}

%\cite{ChangShaoZhang2017}
%\cite{ChangShaoZhang2015}
%\cite{FriezeJerrum1997}

%\

%\bibliography{journalname,graph}

\begin{thebibliography}{BRSH13}

\bibitem[Bac13]{Bach2013}
F.~Bach.
\newblock {Learning with submodular functions: A convex optimization
  perspective}.
\newblock {\em Found. Trends Mach. Learning}, 6:145--373, 2013.

\bibitem[BF08]{BercziFrank2008}
K.~B{\'e}rczi and A.~Frank.
\newblock Variations for {Lov{\'a}sz}' submodular ideas.
\newblock In M.~Gr{\"o}tschel and G.~O.~H. Katona, editors, {\em {Building
  Bridges between Mathematics and Computer Science}}, pages 137--164. Springer,
  2008.

\bibitem[BJ13]{BauerJost2013}
F.~Bauer and J.~Jost.
\newblock Bipartite and neighborhood graphs and the spectrum of the normalized
  graph laplace operator.
\newblock {\em Commun. Anal. Geom.}, 21:787--845, 2013.

\bibitem[BRSH13]{BuhlerRangapuramSetzerHein2013}
T.~B{\"u}hler, S.~S. Rangapuram, S.~Setzer, and M.~Hein.
\newblock Constrained fractional set programs and their application in local
  clustering and community detection.
\newblock In {\em Proceedings of the 30th International Conference on Machine
  Learning}, pages 624--632, 2013.

\bibitem[Cha16]{Chang2015}
K.~C. Chang.
\newblock Spectrum of the $1$-laplacian and {Cheeger}'s constant on graphs.
\newblock {\em J. Graph Theor.}, 81:167--207, 2016.

\bibitem[Chu97]{Chung1997}
F.~R.~K. Chung.
\newblock {\em {Spectral Graph Theory}}.
\newblock American Mathematical Society, revised edition, 1997.

\bibitem[CSZ15]{ChangShaoZhang2015}
K.~C. Chang, S.~Shao, and D.~Zhang.
\newblock {The 1-Laplacian Cheeger cut: Theory and algorithms}.
\newblock {\em J. Comput. Math.}, 33:443--467, 2015.

\bibitem[CSZ16]{ChangShaoZhang2016-maxcut}
K.~C. Chang, S.~Shao, and D.~Zhang.
\newblock {Spectrum of the signless $1$-Laplacian and the dual {Cheeger}
  constant on graphs}.
\newblock {\em arXiv:1607.00489}, 2016.

\bibitem[CSZ17]{ChangShaoZhang2017}
K.~C. Chang, S.~Shao, and D.~Zhang.
\newblock {Nodal domains of eigenvectors for {1-Laplacian} on graphs}.
\newblock {\em Adv. Math.}, 308:529--574, 2017.

\bibitem[CSZZ18]{ChangShaoZhangZhang2018-a}
K.~C. Chang, S.~Shao, D.~Zhang, and W.~Zhang.
\newblock A simple algorithm for {Max Cut}.
\newblock {\em arXiv:1803.06496}, 2018.

\bibitem[FJ97]{FriezeJerrum1997}
A.~Frieze and M.~Jerrum.
\newblock {Improved approximation algorithms for MAX k-CUT and MAX BISECTION}.
\newblock {\em Algorithmica}, 18:67--81, 1997.

\bibitem[HB10]{HeinBuhler2010}
M.~Hein and T.~B\"{u}hler.
\newblock An inverse power method for nonlinear eigenproblems with applications
  in 1-spectral clustering and sparse {PCA}.
\newblock In {\em Advances in Neural Information Processing Systems 23}, pages
  847--855, 2010.

\bibitem[Kar72]{Karp1972}
R.~M. Karp.
\newblock Reducibility among combinatorial problems.
\newblock In R.~E. Miller, J.~W. Thatcher, and J.~D. Bohlinger, editors, {\em
  {Complexity of Computer Computations}}, pages 85--103. Springer, 1972.

\bibitem[Lov83]{Lovasz1983}
L.~Lov{\'a}sz.
\newblock Submodular functions and convexity.
\newblock In A.~Bachem, M.~Gr{\"o}tschel, and B.~Korte, editors, {\em
  {Mathematical Programming: the State of the Art}}, pages 235--257. Springer,
  1983.

\bibitem[Tre12]{Trevisan2012}
L.~Trevisan.
\newblock Max cut and the smallest eigenvalue.
\newblock {\em SIAM J. Comput.}, 41:1769--1786, 2012.

\bibitem[Xu16]{Xu2016}
B.~Xu.
\newblock {Graph partitions: Recent progresses and some open problems}.
\newblock {\em Adv. Math. (China)}, 45:1--20, 2016.

\end{thebibliography}
%\bibliographystyle{alpha}

\end{document}